\documentclass[11pt]{article}
\usepackage{amsmath,amssymb,amsthm,mathrsfs}
\usepackage{graphicx}
\usepackage{xcolor}
\usepackage{booktabs}
\usepackage{multirow}
\usepackage{subcaption}
\usepackage[authoryear,round]{natbib}
\usepackage{hyperref}
\hypersetup{
  colorlinks=true,
  linkcolor=blue,
  citecolor=blue,
  urlcolor=blue
}

\makeatletter
\def\@fnsymbol#1{\ensuremath{\ifcase#1\or *\or \ddagger\or
   \mathsection\or \mathparagraph\or \|\or **\or
   \ddagger\ddagger\else\@ctrerr\fi}}
\makeatother

\theoremstyle{definition}
\newtheorem{lemma}{Lemma}[section]
\newtheorem{proposition}{Proposition}[section]

\newtheorem{definition}{Definition}[section]

\newtheorem{corollary}{Corollary}[section]
\newtheorem{remark}{Remark}[section]

\newcommand{\EE}{\mathbb{E}}
\newcommand{\RR}{\mathbb{R}}

\newcommand{\NN}{\mathbb{N}}

\newcommand{\MM}{\mathbb{M}}
\newcommand{\WW}{\mathbb{W}}

\usepackage{relsize}
\newcommand{\T}{{\mathsmaller {\rm T}}}

\def\vhalf{\textstyle{\frac 1 2}}
\def\vfrac#1#2{{\textstyle{\frac{#1}{#2}}}}
\def\pr{\mathop{\rm pr}\nolimits}
\def\tr{\mathop{\rm tr}\nolimits}
\def\vech{\mathop{\rm vech}\nolimits}

\newcommand{\argmax}{\mathop{\mathrm{argmax}}}

\title{On inferential equivalence classes of causal models}
\author{
  Charlotte Edgar\thanks{Department of Mathematics, Imperial College London} \;\;\quad
  Heather Battey$^*$ 
}
\date{\today}

\begin{document}

\maketitle

\begin{abstract}
				Causal models in the same Markov equivalence class are, in the absence of strong assumptions, statistically indistinguishable at any sample size. For modest sample sizes there is also the possibility that several such classes are compatible with the data, pointing to a confidence set of causal models as the appropriate presentation of evidence. Non-identifiability of Gaussian causal models from the same Markov equivalence class corresponds to a plurality of inverse-covariance representations in the unconstrained parametrisation of \citet{CW1993}. This parametrisation, avoiding conic constraints that would otherwise complicate distributional approximations, facilitates construction of, and theoretical analysis for, a confidence set of causal models based on standard likelihood theory. By drawing on a geometric formulation of \citet{Evans2020}, we provide insight into which causal models are most likely to be included in the confidence set when the true parameter values of the generating model are at the borderline of detectability, delineating those models whose inclusion probabilities are stable at the nominal level, slowly decaying with sample size, and quickly decaying with sample size. We also study settings in which two or more causal models are in operation, exploring how the mixture weights, and the  geometry of the models in the mixture relative to candidate models, interact with the inclusion probabilities. The purpose of the paper is to probe, from the perspective of structural properties of the true causal mechanism, the limits of what is achievable in causal inferential settings.

\bigskip

\noindent \emph{Some key words:} causation; confidence sets of etiological explanations; local geometry; Markov equivalence; mixtures of causal models; model uncertainty.
\end{abstract}

	\section{Introduction}\label{secIntro}

	Establishing causation is often the idealised objective of scientific endeavours even when caution is adopted in the use of the word. In some contexts, it is natural to designate one variable $Y$ as the outcome, all others being prior to $Y$ and thus potentially causal. This is a classical regression setting, where non-zero regression coefficients are suggestive of a causal relationship, any indirect effects of treatments or exposures possibly being explained away by intermediate variables. A conceptually different framework is where there is no a priori reason to suppose that any variable is prior to another, and interest lies in ascertaining the directions of any dependencies between the observable random variables $Y_1,\ldots,Y_d$. Such dependence is often represented as a directed acyclic graph with nodes for random variables, and directed edges for direct causal relationships. 

In the absence of strong assumptions, it is usually not possible to learn a directed acyclic graph from observational data unless it belongs to a singleton Markov equivalence class. For a given generating mechanism, there is a different causal model in the same non-singleton Markov equivalence class that can achieve the same distribution of observable quantities; these are therefore indistinguishable on the basis of physical observations unless context supplies the ordering. There is the possibility, in addition, that several or many equivalence classes are compatible with the data for modest sample sizes, pointing to a confidence set of causal models as the appropriate presentation of evidence. While causal models within a Markov equivalence class are statistically indistinguishable at any sample size, those within a confidence set are statistically indistinguishable from the true generating mechanism at the observed sample size and at the chosen significance level. A confidence set of causal models can therefore be viewed as an inferential equivalence class, an evidential analogue of the Markov equivalence class for the true explanation. 

\citet{Evans2020} showed that the ability to distinguish between two models at small sample sizes is determined by their relative geometry at the point in the parameter space where the models intersect, a favourable geometry implying a faster rate of distinguishability than for models with an unfavourable geometry. We reframe his discussion for pairs of models in terms of probabilities of erroneous causal models being retained in the confidence set. While for sufficiently large sample size, only the Markov equivalence class for the true explanation is retained, our interest here is in low-information settings in which the generating mechanism departs from a completely unconnected graph by a negligible quantity on the borderline of detectability. There is, under this notional regime, a clear separation in the sample-size-dependent probabilities of inclusion in the confidence set for each erroneous causal model. Our construction and analysis is simplified considerably by using a non-standard parametrisation of the Gaussian model due to \citet{CW1993}.

A possibly unexplored but very relevant setting is that in which two or more causal models generate the observable data, different observational units contributing data from different causal mechanisms. For this, the generating process can be expressed as a mixture of causal models on account of the label being unknown. Ideally, the confidence set would include all models in the mixture, but intuitively, only causal models for which the mixture weight is sufficiently large are included with the nominal probability, the permissible weighting being determined by the relative geometry of the constituent models. By formulating a double-asymptotic regime in the mixture weights and signal strengths away from the point at which the models intersect, we probe the extent to which inferential equivalence classes are able to detect both models in a two-component mixture, exploring three-component mixtures by simulation.

\section{Related literature and contributions}\label{secLiterature}

Uncertainty over the statistical model tends not to be emphasised, selection of a single explanation in the light of the data being the more popular and more easily-communicated mode of presentation. That sparsity makes the model itself the object of inference was highlighted in the context of high-dimensional regression by \citet{CB2017}, who emphasised the multiplicity of compatible representations and the dangers of binary decisions in reporting scientific evidence, arguing instead for a confidence set of sparse models. The work reflected earlier comments in the same vein by \citet{Cox1968} and repeated over many years \citep{Cox1977, CoxSnell1974, CoxSnell1989}. The construction of \citet{FY2015} for low-dimensional regression is analogous. These ideas found little traction, but a string of recent papers have pursued closely related directions; see \citet{BRT} for a detailed literature review until 2026. With the exception of \citet{Drton2026}, none of this literature concerns causal models. \citet{Drton2026} proposed a bootstrap construction of a confidence set of causal orderings and established its asymptotic validity. 

The contributions of the present paper principally concern perspective and understanding rather than methodology, but we note as part of this development a simpler construction than that used by \citet{Drton2026}, based on a likelihood-ratio assessment in an unconstrained parametrisation of \citet{CW1993}. The parametrisation is such that sparsity induced by missing edges respects positive definiteness, thereby avoiding conic constraints and related difficulties. The simplicity of the formulation means that relatively standard distributional approximations can be applied, allowing the resulting confidence sets of causal models to be connected to the geometric equivalence classes of \citet{Evans2020}. 

As far as we are aware, the importance of mixtures of causal models has not been emphasised in the statistical literature. An indirect reference to their importance was made by the sociologist \citet{Goldthorpe2025} in connection with the triangulation of findings from different types of study:
\begin{quote}
	``\emph{Such research may be of particular importance where $Y$ is seen as resulting not from a single process but from several different causal processes, which need to be disentangled and their relative importance assessed.}'' \citep[][p.6]{Goldthorpe2025}
\end{quote}
Triangulation of evidence is an important open problem, addressed only briefly in \S \ref{secDiscuss}, but in \S \ref{secMixtures} we explore the extent to which a confidence set of causal models is capable of recovering all explanations when two or more causal models are in operation via a convex mixture.

Proofs in the supplementary material are based on established ideas but are not a direct application of existing results, as the confidence-set perspective treats models asymmetrically via the usual analogue of proof by contradiction. Instead, the geometric insights of \citet{Evans2020} need to be embedded within the classical analysis of likelihood-ratio tests under contiguous alternatives, best approached via \citet{vdV2000}.

\section{Some initial insights in three variables}\label{secInitInsight}

We start by providing some initial insight into the relative distinguishability of simple causal models in three variables and how this is reflected in the behaviour of the confidence sets. The details of implementation are not at this point important, and are deferred to \S \ref{secConstruction}. More general formal statements are given in \S \ref{secCalibration}.

Consider three variables $X$, $Y$, and $Z$. Assuming that there is one missing directed edge, there are 12 possible V-configurations, represented in Table \ref{tabVConfig} alongside the corresponding causal ordering. These are grouped by Markov equivalence class, the labelling of classes, as well as the labelling of models, being otherwise arbitrary.

\begin{table}[h]
	\begin{center}
		\begin{tabular}{|cccc|}
			\hline
			\multirow{2}{*}{model} & \multirow{2}{*}{equivalence class} & \multirow{2}{*}{causal ordering} & \multirow{2}{*}{graphical representation} \\
			&                                           &                                  &                                           \\
			1     & 1 & $\{X,Z\}, \{Y\}$      & $X \rightarrow Y  \leftarrow Z$ \\
			&   &                       &                                 \\
			2     & 2 & $\{Y\}, \{X, Z\}$     & $X \leftarrow Y \rightarrow Z$  \\
			3     & 2 & $\{Z\},\{Y\}, \{X\}$  & $X \leftarrow Y \leftarrow Z$ \\
			4     & 2 & $\{X\}, \{Y\}, \{Z\}$ & $X \rightarrow Y \rightarrow Z$ \\
			&   &                       &                                 \\
			5     & 3 &    $\{Y,Z\}, \{X\}$   & $Y \rightarrow X \leftarrow Z$ \\
			&   &                       &                                 \\
			6     & 4 & $\{X\}, \{Y,Z\}$      & $Y \leftarrow X \rightarrow Z$  \\
			7     & 4 & $\{Z\},\{X\}, \{Y\}$  & $Y \leftarrow X \leftarrow Z$   \\
			8     & 4 & $\{Y\}, \{X\}, \{Z\}$ & $Y \rightarrow X \rightarrow Z$ \\
			&   &                       &                                 \\
			9     & 5 &   $\{X,Y\}, \{Z\}$    & $X \rightarrow Z \leftarrow Y$ \\
			&   &                       &                                 \\
			10    & 6 &   $\{Z\}, \{X,Y\}$    & $X \leftarrow Z \rightarrow Y$ \\
			11    & 6 &  $\{Y\},\{Z\}, \{X\}$ & $X \leftarrow Z \leftarrow Y$   \\
			12    & 6 & $\{X\}, \{Z\}, \{Y\}$ & $X \rightarrow Z \rightarrow Y$ \\
			\hline                                       
		\end{tabular}
	\end{center}
	\medskip
	\caption{Causal dependencies in 3 variables and their Markov equivalence classes.\label{tabVConfig}}
\end{table}

Let $a_{ij}$ denote the strength of the effect of variable $i$ on variable $j$. When all such non-zero effects are sufficiently strong, or when the sample size $n$ is sufficiently large, a valid $\alpha$-level confidence set of causal models must include the true model with probability $1-\alpha$. By \emph{true model} we mean the model containing the true generating mechanism as described by the relevant row of Table \ref{tabVConfig}, the model accounting for other aspects including the strength of the associations and the distribution of the outcomes. There is, corresponding to each \emph{true model}, a \emph{true Markov equivalence class} determined by the value in the second column. Under the Gaussianity assumption of the present paper, the confidence set includes, by construction unless further identifying assumptions are made, all models in the true Markov equivalence class with the same inclusion probability as for the true model. As the signal to noise ratio or sample size is decreased, other models may also be present in the confidence set and those that are represented with highest $n$-dependent probability are all members of an equivalence class that intersect tangentially, rather than transversally, in a sense made rigorous by \citet{Evans2020} and reviewed in \S \ref{secCalibration}. Figure \ref{figCoverage} illustrates this through simulation for the twelve models in Table \ref{tabVConfig}.

\begin{figure}
	
	\begin{center}
		
		\includegraphics[trim=0.082in 0.1in 0.2in 0.2in, clip,width=0.33\linewidth]{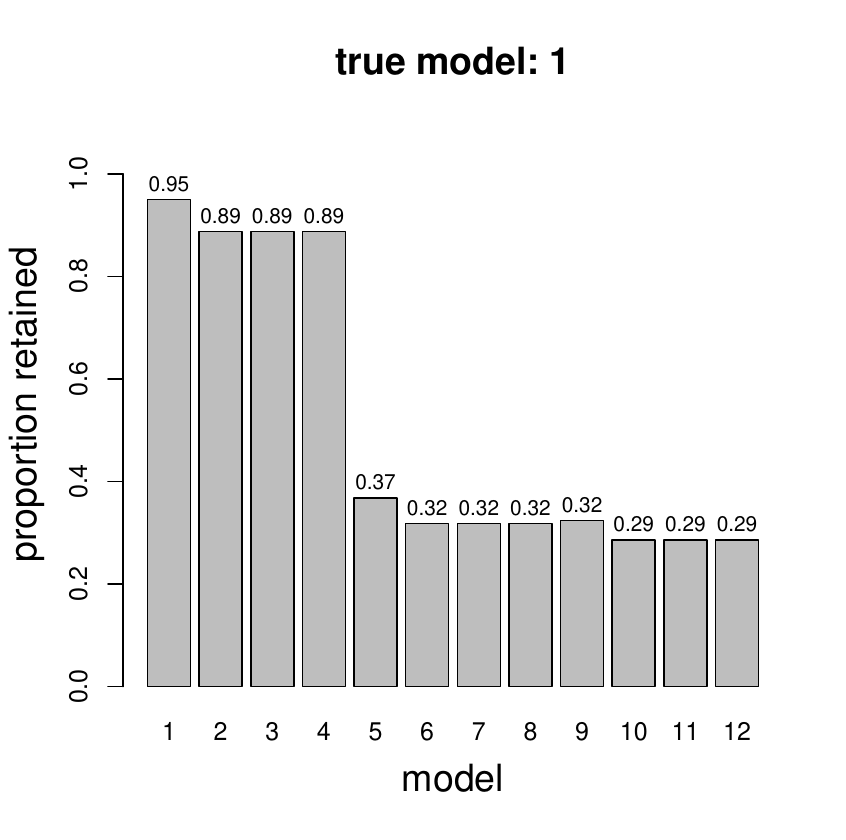}
		\includegraphics[trim=0.31in 0.1in 0.2in 0in, clip,width=0.32\linewidth]{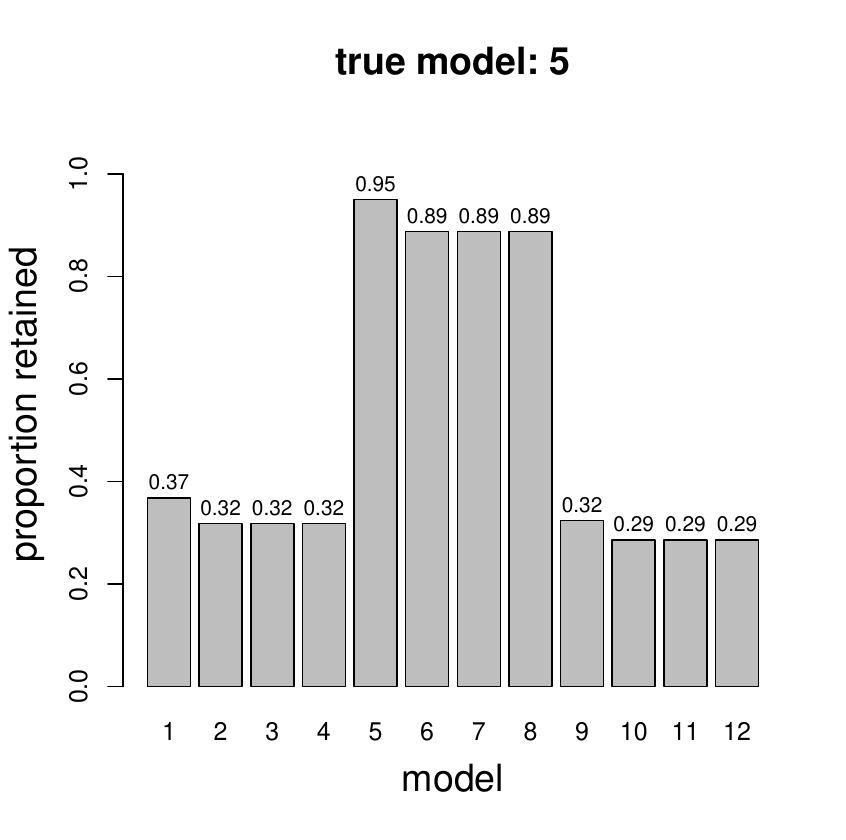} 
		\includegraphics[trim=0.31in 0.1in 0.2in 0in, clip,width=0.32\linewidth]{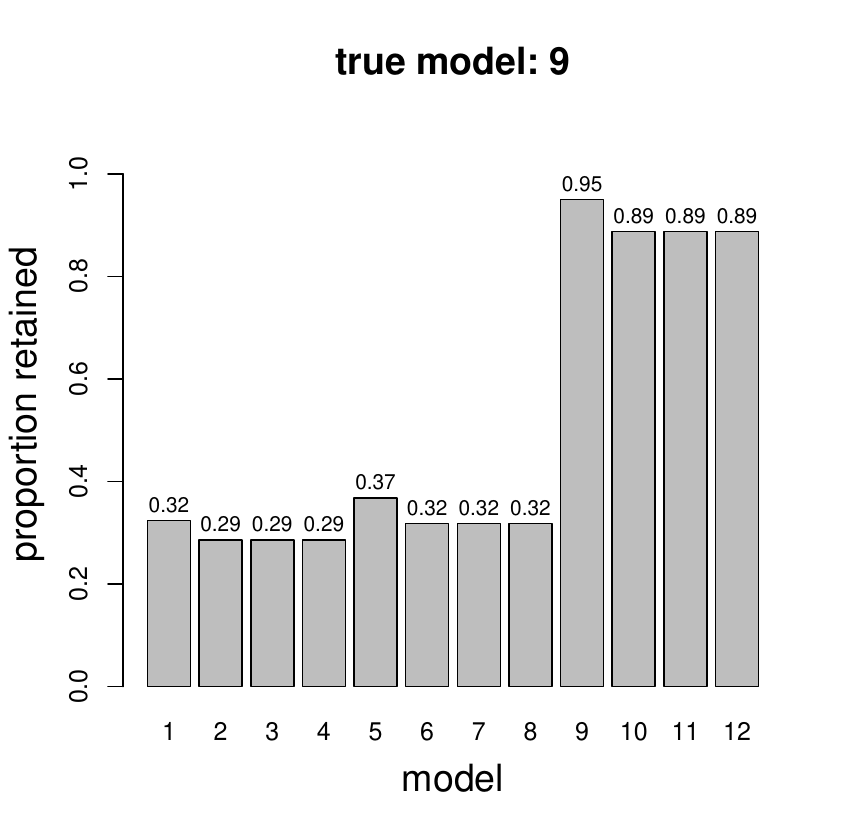}

		\includegraphics[trim=0.082in 0.1in 0.2in 0.2in, clip,width=0.33\linewidth]{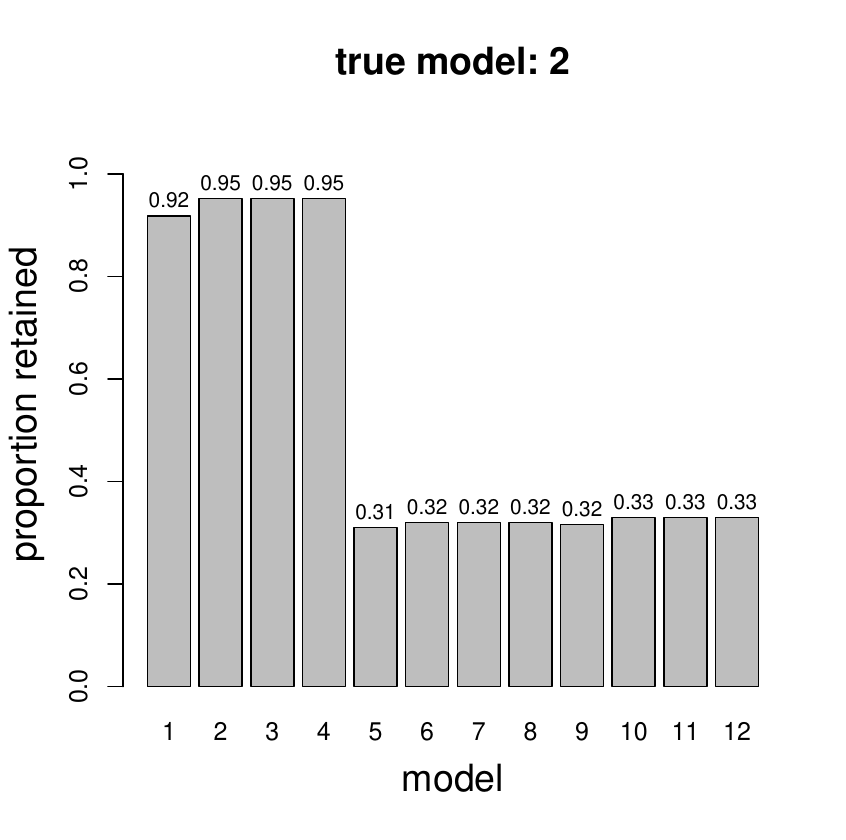} 
		\includegraphics[trim=0.31in 0.1in 0.2in 0in, clip,width=0.32\linewidth]{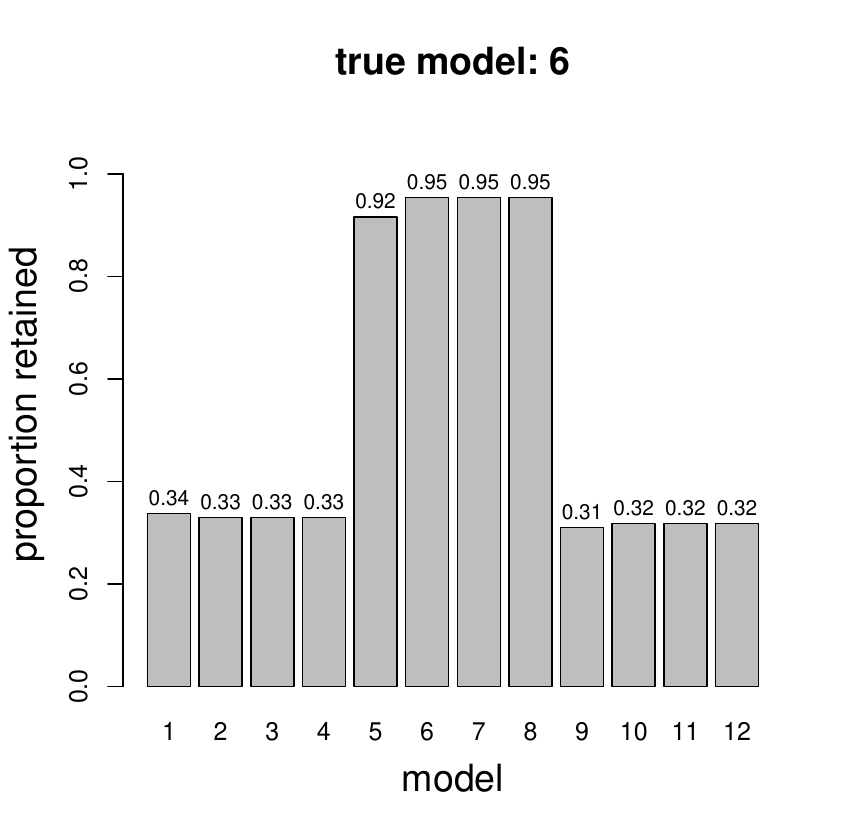} 
		\includegraphics[trim=0.31in 0.1in 0.2in 0in, clip,width=0.33\linewidth]{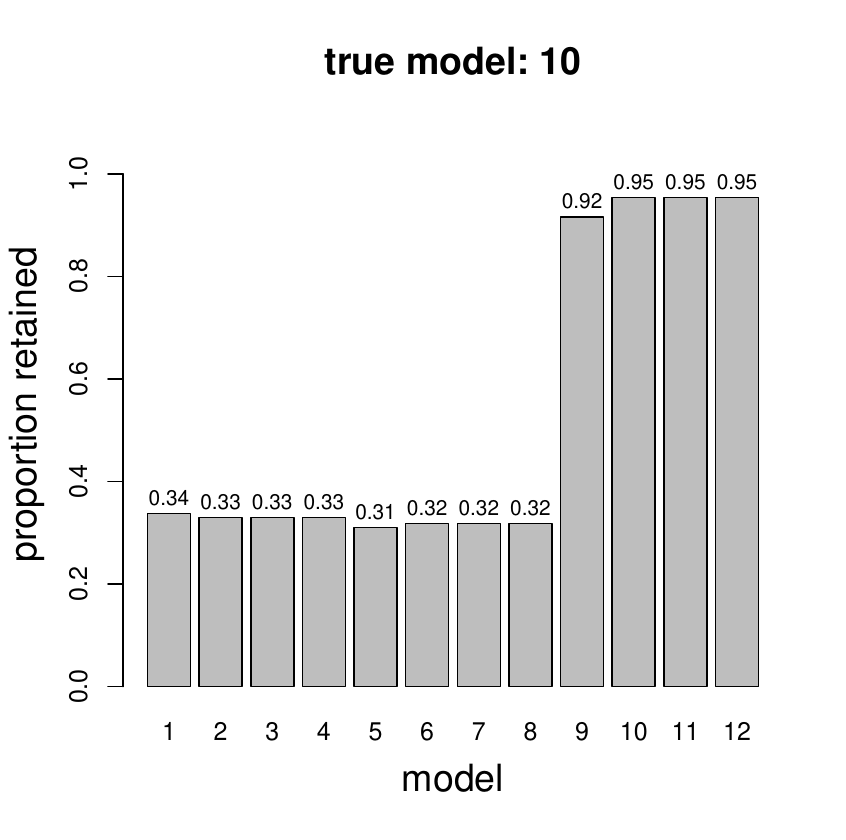}

		\includegraphics[trim=0.082in 0.1in 0.2in 0.2in, clip,width=0.33\linewidth]{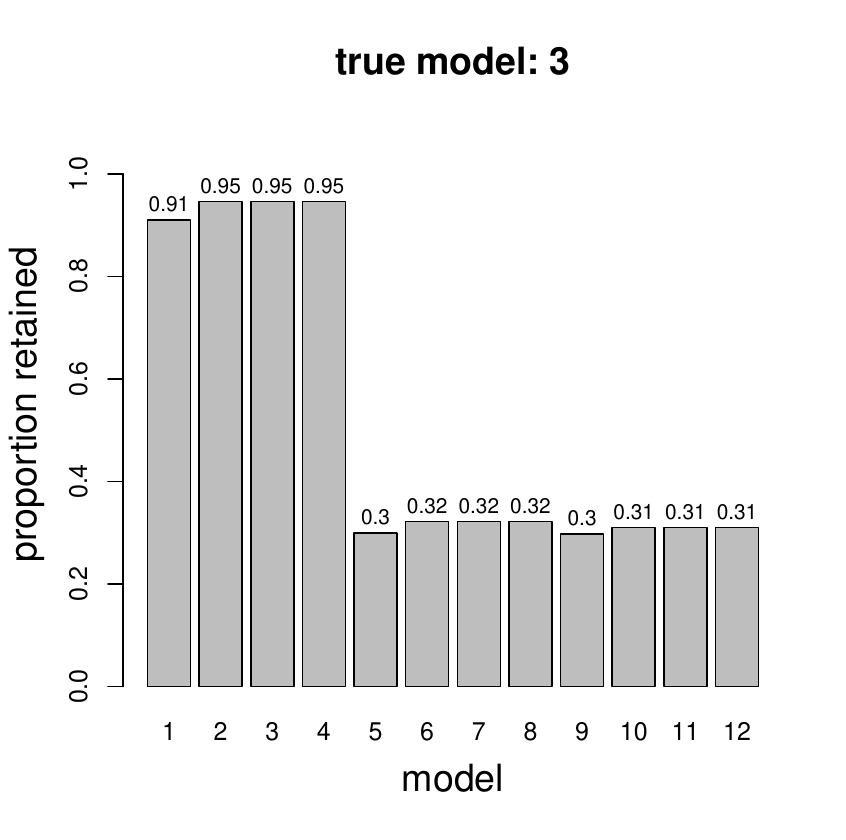} 
		\includegraphics[trim=0.31in 0.1in 0.2in 0in, clip,width=0.33\linewidth]{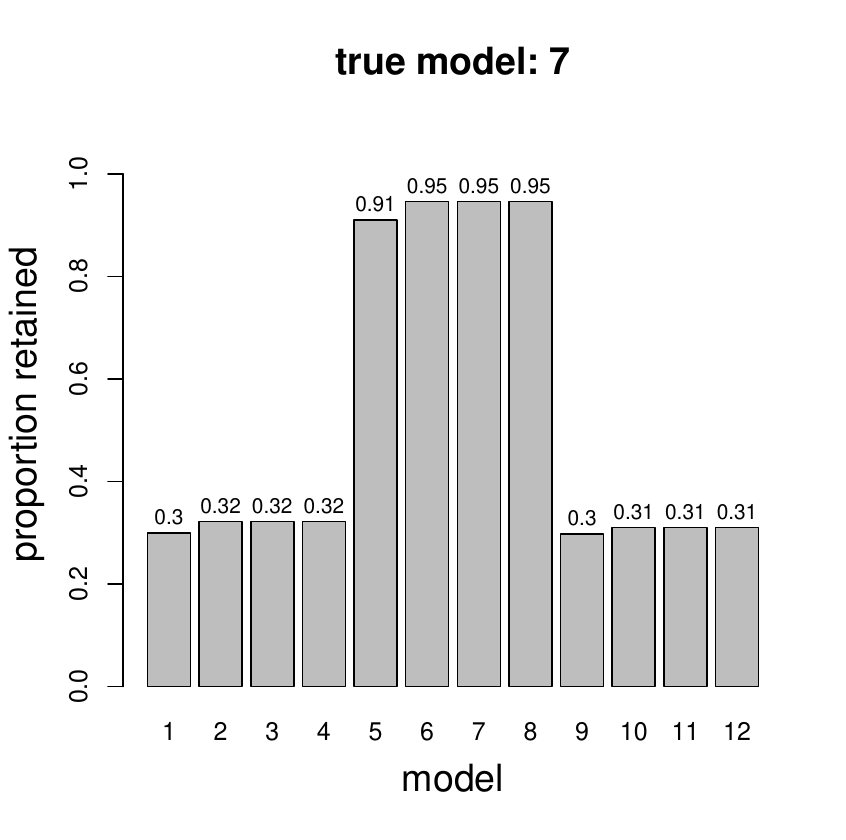}
		\includegraphics[trim=0.31in 0.1in 0.2in 0in, clip,width=0.32\linewidth]{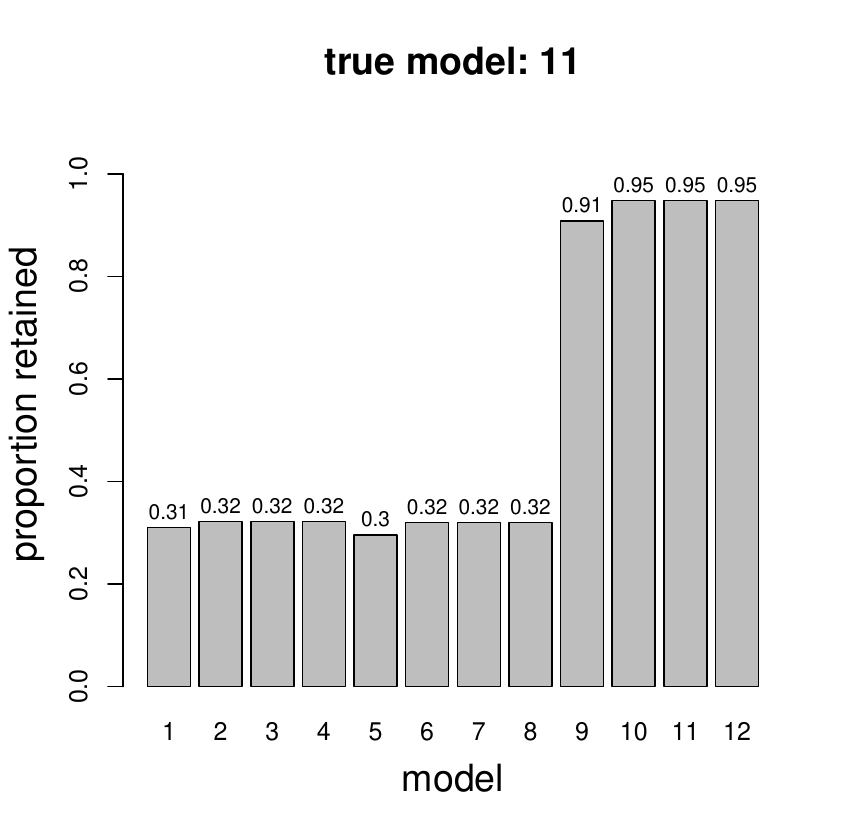} 		
		
		\includegraphics[trim=0.082in 0.1in 0.2in 0.2in, clip,width=0.33\linewidth]{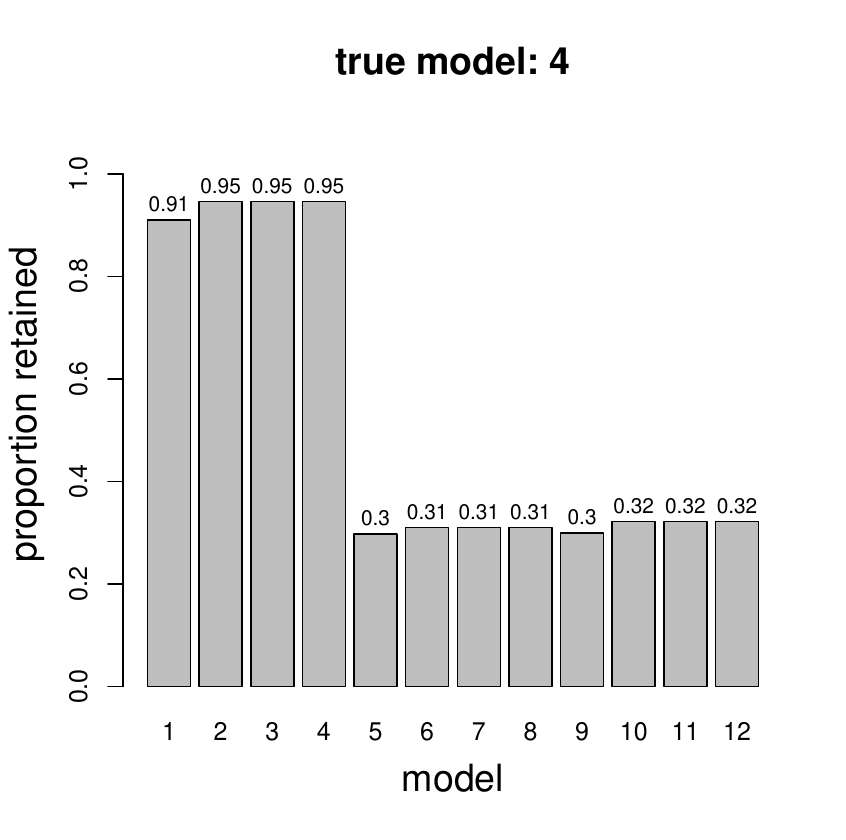} 
		\includegraphics[trim=0.31in 0.1in 0.2in 0in, clip,width=0.32\linewidth]{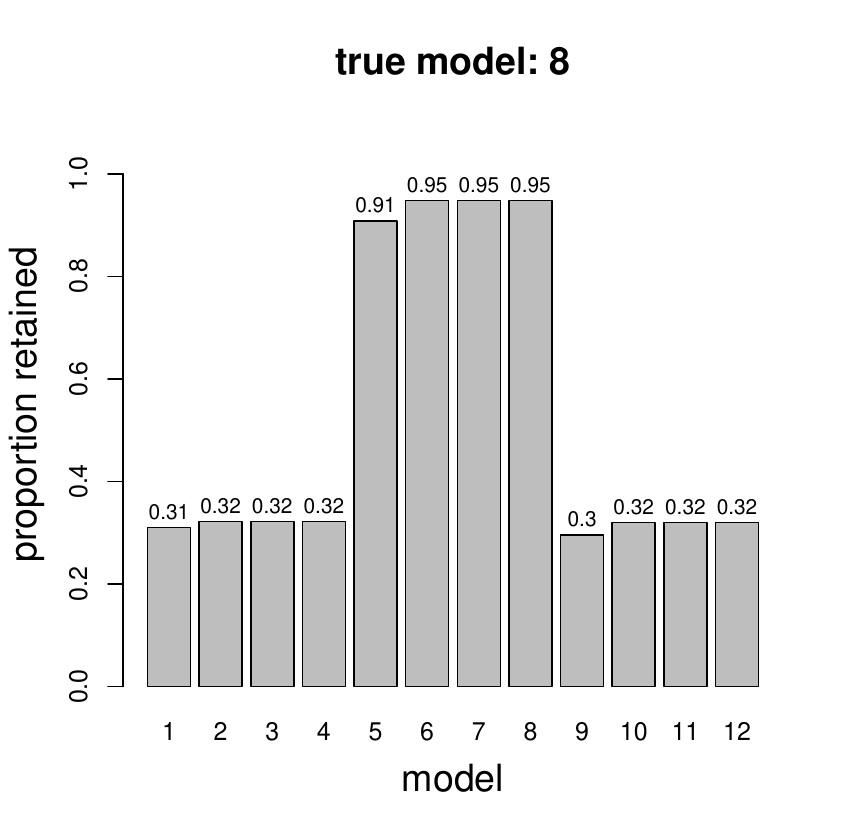}
		\includegraphics[trim=0.31in 0.1in 0.2in 0in, clip,width=0.32\linewidth]{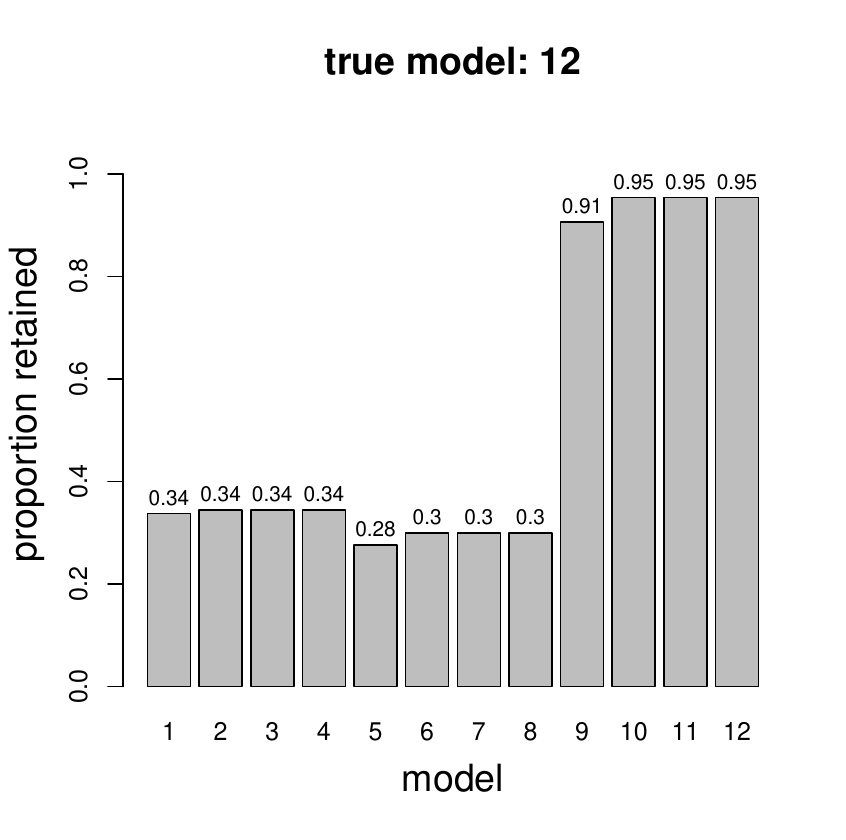} 
		
	\end{center}   
	\caption{Empirical probabilities of inclusion of each model in the $95\%$ confidence set of models when the true model is as specified. The sample size is $n=100$.\label{figCoverage}}
	
\end{figure}

\section{Construction of confidence sets of causal models}\label{secConstruction}

While the construction of \citet{Drton2026} is more involved, there is a simple approach to constructing confidence sets of causal models in a Gaussian setting based on a large number of likelihood-ratio assessments against a suitable encompassing model. An important aspect of this formulation, avoiding theoretical difficulties due to conic constraints, is the use of an unconstrained parametrisation due to \citet{CW1993}.

We consider graphs that are directed and without cycles, and attach to this formulation a Gaussian model for the outcomes $Y=(Y^1,\ldots, Y^d)^\T$, assumed mean-zero. This implies that that there is some ordering of the variables, described by a permutation matrix $P$, such that $LPY=P\varepsilon$, where $L$ is a lower-triangular matrix with ones along the diagonal and $\varepsilon$ is a mean-zero Gaussian random vector with diagonal covariance matrix $D=e^\Delta$. The off-diagonal entries of $A=P^{\T}LP$ are, in the notation of \citet{WermuthCox2004},
\begin{equation}\label{eqABeta}
	a_{ij}=-\beta_{ij.\text{par}_i\backslash j},
\end{equation}
i.e., the negative regression coefficient of variable $j$ in a regression of variable $i$ on variable $j$ and any other variables that appear earlier in the causal ordering. These causal predecessors are often referred to as parents of variable $i$, notated as $\text{par}_i$ in \eqref{eqABeta}.

The assumption that the underlying graph is acyclical implies that $\Sigma$, $\Sigma^{-1}$ and $(A,\Delta)$ specify three parametrisations of the distribution of $Y$ in $d(d+1)/2$ degrees of freedom once the ordering of the variables is fixed, the relation between the three being
\begin{align*}
	\Sigma = & \; A^{-1} e^\Delta (A^{-1})^\T, \\
	\Sigma^{-1} = & \; A^{\T} e^{-\Delta} A.
\end{align*}
While $\Sigma^{-1}\rightarrow (A, \Delta)$ is a valid reparametrisation once the ordering is fixed, for an unspecified causal ordering, the non-zero entries of $A$ could be in any off-diagonal position compatible with acyclicity, i.e.~the map is not injective. Thus, the representation $A^{\T} e^{-\Delta} A \mapsto \Sigma^{-1}(A,\Delta)$ is over-parametrised in the absence of a causal ordering, explaining non-identifiability up to Markov equivalence. For a fixed causal ordering, the $(A, \Delta)$ parametrisation has the advantage of being unconstrained, i.e.~for any given model specified in the $(A, \Delta)$ parametrisation, the non-zero entries can vary arbitrarily without violating positive definiteness of $\Sigma$ or $\Sigma^{-1}$. 

Let $(Y_i)_{i=1}^n$ be independent and identically distributed copies of $Y$. The log-likelihood function, in exponential-family canonical form, depends on the data through the sufficient statistic $S_n=n^{-1}\sum_{i=1}^n Y_i Y_i^\T$ as
\begin{equation}\label{eqLogLikSig}
	\ell(\Sigma^{-1}; S_n) = \frac{n}{2}\{\log \det(\Sigma^{-1}) - \text{tr}(S_n \Sigma^{-1})\}.
\end{equation}
By the theory of exponential families, or by direct matrix differentiation, the maximum likelihood estimator $\hat\Sigma^{-1}$, if it exists, sets the sufficient statistic equal to its expected value, the solution to the resulting likelihood equation being $\hat\Sigma^{-1}=S_n^{-1}$ provided that $S_n$ is positive definite.

Under any hypothesised causal model, the parametrisation $A^{\T} e^{-\Delta} A \mapsto \Sigma^{-1}(A,\Delta)$ is injective, and the non-zero entries of $A$ and $\Delta$, being unconstrained, present no difficulties for maximisation of the log-likelihood function,
\begin{equation}\label{eqLogLikA}
	\ell(A, \Delta; S_n) = \frac{n}{2}\{\log \det(A^{\T} e^{-\Delta} A) - \text{tr}(S_n A^{\T} e^{-\Delta} A)\}.
\end{equation}
The log-likelihood ratio statistic for assessing a given causal model $m_0$ with corresponding $A_0$ is thus
\begin{equation}\label{eqLogLikRatio}
	\Lambda_0 = 2\{\ell(\hat{\Sigma}^{-1};S_n) - \ell(\hat{A}_0, \hat{\Delta}_0;S_n)\},
\end{equation}
where $(\hat{A}_0, \hat{\Delta}_0)$ maximises \eqref{eqLogLikA} over entries of $A$ and $\Delta$ that respect the zero constraints specified by model $m_0$. There is a slight notational imprecision in \eqref{eqLogLikRatio}, in that the evaluation points of the log-likelihood function are given in different parametrisations. We could have alternatively written \eqref{eqLogLikRatio} in terms of $\hat\Sigma_0^{-1} = \hat{A}_0^\T e^{-\hat{\Delta}_0} \hat{A}_0$. 

Suppose that the true causal model, $m_0$ say, is specified by the positions of the zero entries in $A_0$. By standard likelihood theory, $\Lambda_0$ from equation \eqref{eqLogLikRatio} converges in distribution to a $\chi^2_{\nu}$ random variable, where $\nu$ is the number of zero constraints in $A_0$ relative to an unconstrained model in $d(d+1)/2$ unknown parameters. For each hypothesised model, the corresponding likelihood ratio statistic is calibrated against this $\chi^2_\nu$ distribution, extremity pointing to a contradiction of the model with the data, in the relevant statistical sense. All causal models not rejected at level $\alpha$ comprise a confidence set of causal models, or equivalently a confidence set of Markov equivalence classes.

What is unclear from the standard theory is which erroneous causal models, other than those in the same Markov equivalence class as $m_0$, are likely to also be retained in the confidence set at small sample sizes. That there are differences in retention probabilities is apparent from Figure \ref{figCoverage}. 
Section \ref{secCalibration} establishes the relevant theory, embedding the geometric perspectives of \citet{Evans2020} into the local likelihood-ratio theory of \citet{vdV2000}.

\section{Inferential and geometric equivalence classes}\label{secCalibration}

\subsection{Notation and definitions}

We extend the above notation by specifying for any causal model $m_j$ the constrained maximum likelihood estimator
\begin{equation}\label{eqPrecParam}
	\hat{\Sigma}^{-1}_j = \hat{A}_j^{\T} e^{-\hat\Delta_j} \hat A_j.
\end{equation} 
We will always take $m_0$ to be the true model, thus if model $m_1$ is assumed, the estimated matrices on the right hand size of \eqref{eqPrecParam} will converge to matrices $(A_1^0, D_1^0)$, typically not equal to $(A_0, D_0)$, where $D_0 = e^{-\Delta_0}$. If $m_1$ is Markov equivalent to $m_0$, $\hat \Sigma_1^{-1}$ converges to the true precision matrix $\Sigma_0^{-1}$, as $m_1$ induces the same likelihood function in the $\Sigma^{-1}$ parametrisation. We use the condensed notation $\gamma=\text{vech}(\Sigma)\in \RR^{d(d+1)/2}$, where $\text{vech}(\cdot)$ is the half-vectorisation map for a symmetric matrix. As before, $\hat\gamma_j$ corresponds to the maximum likelihood estimate under model $m_j$, with corresponding true value $\gamma_j$, thus $\gamma_j=\gamma_k$ and $\hat\gamma_j=\hat\gamma_k$ if and only if $m_j \sim m_k$, where $\sim$ denotes Markov equivalence. 

\citet{Evans2020} introduced the following notion of $c$-equivalence between pairs of statistical models. The condensed notation $\gamma^*\in m_0\cap m_1$ here means that a distribution parametrised by a vector $\gamma$ belongs to both models when evaluated at $\gamma^*$. In other words $\gamma^*$ characterises the point in the parameter space at which the the two models intersect. In principle there could be a countable or uncountable set of such values $\gamma^*$, but in the regular settings covered in the present paper, $\gamma^*$ is a vector in a singleton set.

\begin{definition}[$c$-equivalence \citep{Evans2020}]\label{defCEquiv}
	Models $m_0$ and $m_1$ are said to be $c$-equivalent at the point of intersection $\gamma^*\in m_0\cap m_1$ if
	\begin{equation}\label{eqLim}
		\lim_{\varepsilon \downarrow 0} \varepsilon^{-c} d_H\bigl(m_0 \cap N_{\varepsilon}(\gamma^*), m_1 \cap N_{\varepsilon}(\gamma^*)\bigr)=0,
	\end{equation}
	where $N_{\varepsilon}(\gamma^*)$ is a Euclidean ball of radius $\varepsilon$ centred at $\gamma^*$ and $d_H(A,B)$ is the Hausdorff distance between sets, given by
	\[
	d_H(A,B) = \max\biggl\{\sup_{a\in A} \inf_{b\in B}\|a-b\|_{2}, \sup_{b\in B} \inf_{a\in A}\|a-b\|_{2}\biggr\}.
	\]
\end{definition}

\begin{figure}
	\begin{center}
		\includegraphics[trim=0in 0.5in 0in 0in, clip,width=0.9\linewidth]{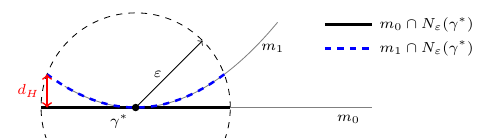}
	\end{center}        
	\caption{Two simplified models $m_0=\{(\gamma,0): \gamma \in \RR\}$ and $m_1=\{(\gamma, \vhalf \gamma^2): \gamma\in\RR\}$ intersecting at $\gamma^*$. The Hausdorff distance constrained the the Euclidean ball is $d_H=\vhalf \varepsilon^2=o(\varepsilon)$. The model is 1-equivalent and 2-near equivalent.\label{figCEquiv}}
\end{figure}

Definition \ref{defCEquiv} says that models $m_0$ and $m_1$ are $c$-equivalent at $\gamma^*$ if the distance between the sets in a ball of radius $\varepsilon$ is $o(\varepsilon^c)$. A weaker definition that is easier to verify and visualise is $c$-near equivalence, in which the $o(\varepsilon^c)$ is replaced by $O(\varepsilon^c)$. Equivalently the zero limit in \eqref{eqLim} is replaced by a finite constant. Because $O(\varepsilon^{c+1})=o(\varepsilon^{c})$ for $\varepsilon\rightarrow 0$, $(c+1)$-near equivalence implies $c$-equivalence, but the converse is not true in general. As noted by \citet{Evans2020}, however, $(c-1)$-equivalence implies $c$-near equivalence if the models are $k$-times differentiable and $c\leq k-1$. Figure \ref{figCEquiv} depicts a simple example of two 2-near equivalent models. 

For the 12 causal models in Table \ref{tabVConfig}, consider the set of Gaussian distributions satisfying the independence constraints implied by each model. The skeletons of each directed acyclic graph are obtained by replacing directed edges with undirected ones giving $X-Y-Z$, $Y-X-Z$ and $X-Z-Y$, and two Markov equivalence classes corresponding to each skeleton. The singleton Markov equivalence class for each skeleton consists of the model with a sink node as the central node, meaning two incoming arrows. Models in different equivalence classes that share the same skeleton are 2-near-equivalent, as illustrated by \citet[][Example 2.4]{Evans2020} and formalized in his Theorem 4.2. Models with different skeletons are 1-near-equivalent.

\subsection{Local geometry and properties of the confidence set} 

The purpose of this section is to probe the limits of what is achievable from the confidence set construction, relating these to the structural properties of the model to which the data generating process belongs.

For sufficiently large sample size $n$, only the models that are Markov equivalent to $m_0$ will be contained in the confidence set at the nominal probability $1-\alpha$. To provide insight for small $n$, we start by considering notional contiguous models $m_1^{(n)}$ whose strength of departure from $m_0$ is at the borderline of detectability, by which we mean shrinking at the same rate as the standard parametric estimation error. More formally, let the parameter vectors characterising $m_0^{(n)}$ and $m_1^{(n)}$ converge to $\gamma^*\in m_0 \cap m_1$ as $n\rightarrow \infty$. By construction, this implies that all elements of $m_1^{(n)}$ have parameters converging to the true parameter $\gamma_0^{(n)}$, i.e.~$m_1^{(n)}$ consists of contiguous alternatives in the sense of \citet[][ch.~7]{vdV2000}. 
Proposition \ref{propCSProperties} restates the insights of \citet[][Corollary 2.11]{Evans2020} in terms of properties of the confidence set of models.

\begin{proposition}\label{propCSProperties} 
	Let $m_0^{(n)}$ be the true model with parameter $\gamma_0^{(n)}$ in the interior of its parameter space and let $m_1^{(n)}$ be a contiguous postulated model, specifying local departures from $\gamma_0^{(n)}$ in the direction of $h_1\in \RR^{d(d+1)/2}$. Specifically, the permissible parameters under model $m_1^{(n)}$ are $\gamma_1^{(n)} = \gamma_0^{(n)} + h_1/\sqrt{n}$. In terms of the point of intersection, $\gamma_0^{(n)}=\gamma^* + a_0^{(n)}$ and $\gamma_1^{(n)}=\gamma^* + a_1^{(n)}$, so that $h_1=\sqrt{n}(a_1^{(n)}-a_0^{(n)})$. Let $\mathcal{M}$ be the set of all models $m_1$ not rejected at level $\alpha$ by a likelihood ratio test of $m_1$ against the unconstrained model. As $n\rightarrow \infty$:
	\begin{enumerate}
		\item[(i)] All models $m_1^{(n)}\sim m_0^{(n)}$ are included in $\mathcal{M}$ with probability $1-\alpha$ irrespective of $a_0^{(n)} \asymp a_1^{(n)}$.
		\item[(ii)] All models $m_1^{(n)}$ that are $c$-equivalent to $m_0^{(n)}$ at $\gamma^*$ are included in $\mathcal{M}$ with probability $1-\alpha+o(1)$ if and only if $a_0^{(n)}\asymp a_1^{(n)} \lesssim n^{-1/2c}$. Otherwise they are included with probability $o(1)$.
	\end{enumerate}
	Moreover, if any single model in a Markov equivalence class is included in $\mathcal{M}$, then all models in the same class are also included.
\end{proposition}

\begin{remark}
	The vector-valued sequences $a_0^{(n)}$ and $a_1^{(n)}$ potentially have different entries but since it is their qualitative behaviour with increasing $n$ that is relevant, the same notation will be used to refer to the vectors themselves and the rate of convergence of their entries, context resolving ambiguity. 
\end{remark}

Proposition \ref{propCSProperties} is stated most conveniently in terms of $c$-equivalence, but $c$-near-equivalence has a more direct geometric interpretation. Since $(c-1)$-equivalence implies $c$-near-equivalence, all models $m_1^{(n)}$ that are $c$-near-equivalent to $m_0^{(n)}$ at $\gamma^*$ are included in $\mathcal{M}$ with probability $1-\alpha + o(1)$ if and only if $a_0^{(n)} \asymp a_1^{(n)}= O(n^{1/2(c-1)})$. In other words, $c$-near-equivalent models are indistinguishable from $m_0^{(n)}$ if and only if $a_0^{(n)} \asymp a_1^{(n)} = o(n^{-1/2c})$. 

Proposition \ref{propCSProperties} sheds light on how the local geometry of the two models interacts with the behaviour of the confidence set. It is, however, more natural in the present context, where the two models are not on an equal footing, to consider a single asymptotic regime in which only the parameters associated with the true model $m_0$ are considered as notionally approaching the point of intersection $\gamma^*$. In other words, no constraints are imposed in the fitting of $m_1$ that would force it to be contiguous to $m_0$. Proposition \ref{propKL} shows that the Kullback-Leibler minimising parameter inherits its behaviour from $a_0^{(n)}$, pointing to a sharp threshold. For this we define $\mathbb{M}_n(\gamma):= \ell(\gamma)/n$ to be the rescaled log-likelihood function and $\mathbb{M}_0^{(n)}(\gamma):=\EE\mathbb{M}_n(\gamma)$, which depends on $n$ via the expectation operator under the sequence of true parameter values $\gamma_0^{(n)}=\gamma^* + a_0^{(n)}$. Let $\Gamma_0$ and $\Gamma_1$ be the parameter spaces under models $m_0$ and $m_1$ respectively. 

\begin{proposition}\label{propKL}
	Let $\hat{\gamma}_1^{(n)} :=  \argmax_{\gamma \in \Gamma_1} \mathbb{M}_n(\gamma)$ be the maximum likelihood estimator assuming $m_1$ and let $\gamma_1^0 :=  \argmax_{\gamma \in \Gamma_1} \mathbb{M}_0^{(n)}(\gamma)$, the Kullback-Leibler minimising projection. Under a sequence of true parameter values of the form $\gamma_0^{(n)}=\gamma^* + a_0^{(n)}$, the Kullback-Leibler projection satisfies $\gamma_1^0 = \gamma^* + O(a_0^{(n)})$ and therefore $\hat{\gamma}_1^{(n)} = \gamma^* + O_p(\max\{n^{-1/2}, a_0^{(n)}\})$.
\end{proposition}

The rate of convergence of $\gamma_1^0$ to $\gamma^*$ along with the usual $O(n^{-1/2})$ fluctuations induced through estimation can be viewed as an implicit constraint on the model class $m_1$ analogous to $m_1^{(n)}$ from Proposition \ref{propCSProperties}. Thus, Propositions \ref{propCSProperties}  and \ref{propKL} can be interpreted jointly. Suppose that $a_0^{(n)}\gtrsim n^{-1/2}$, i.e.~the signal decays more slowly than the estimation error. Then both $\gamma_0$ and $\hat\gamma_1$ are in a ball of radius $a_0^{(n)}$ centred at $\gamma^*$. From Proposition \ref{propCSProperties}, models $m_1 \not\sim m_0$ are excluded from the confidence set if and only if $a_0^{(n)}\gtrsim n^{-1/2c}$ where $c$ is the $c$-equivalence class of $m_1$ relative to $m_0$. In other words, if the signal decays too sharply relative to the estimation error then there is no hope of detecting $m_1 \not\sim m_0$ as erroneous. The permissible rate of decay decreases with increasing $c$, i.e.~a larger $c$ requires a slower rate $a_0^{(n)}$.

\begin{figure}[h!]
	\begin{center}
		\includegraphics[trim=0in 0in 0in 0in, clip,width=0.45\linewidth]{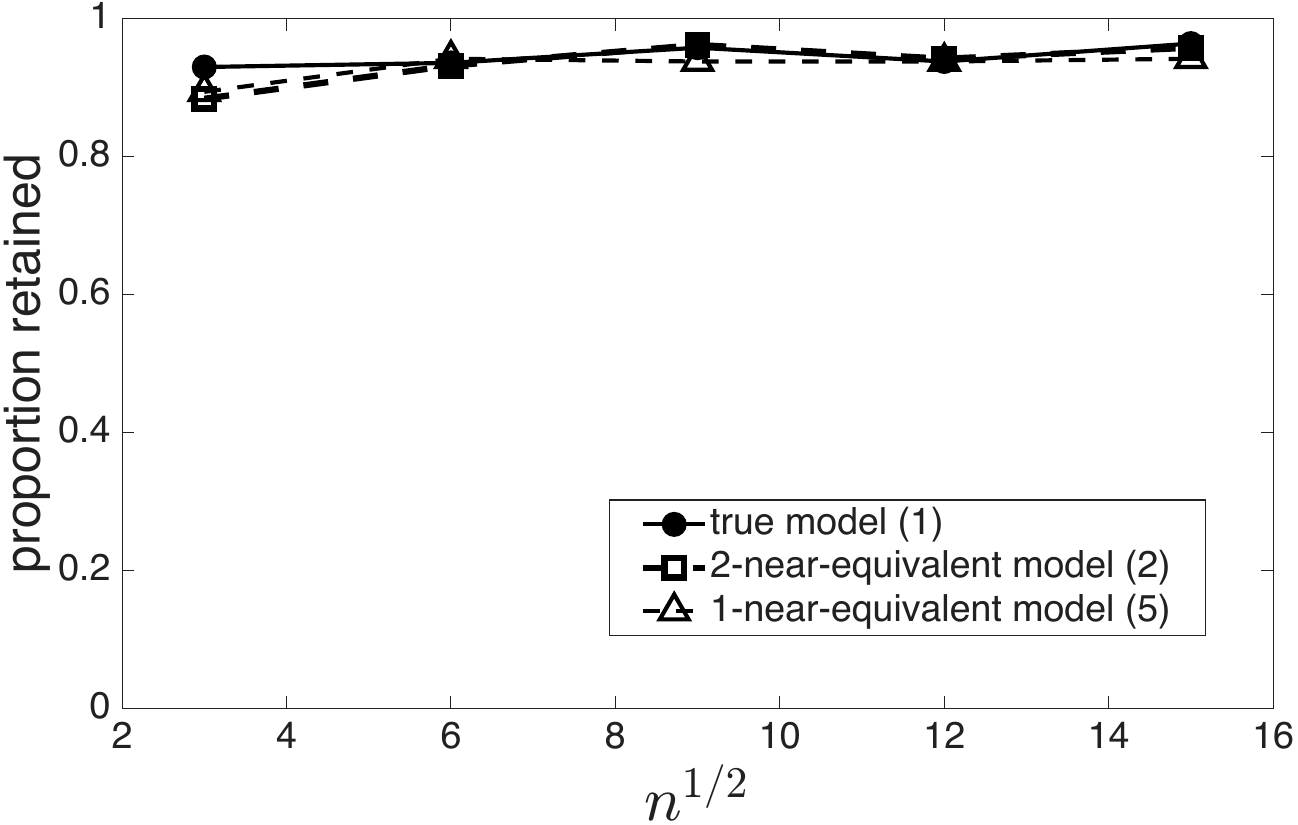}
		\includegraphics[trim=0in 0in 0in 0in, clip,width=0.45\linewidth]{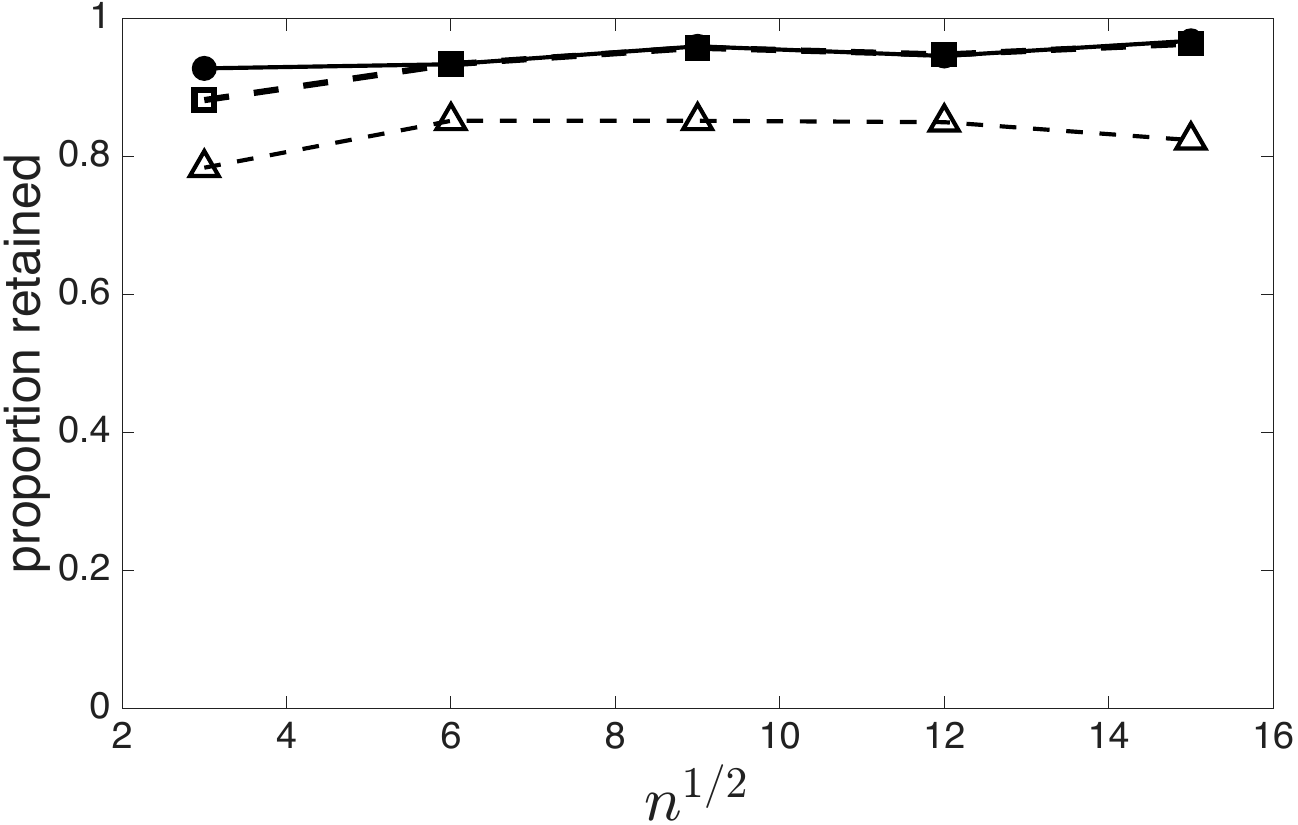} 
		
		\includegraphics[trim=0in 0in 0in 0in, clip,width=0.45\linewidth]{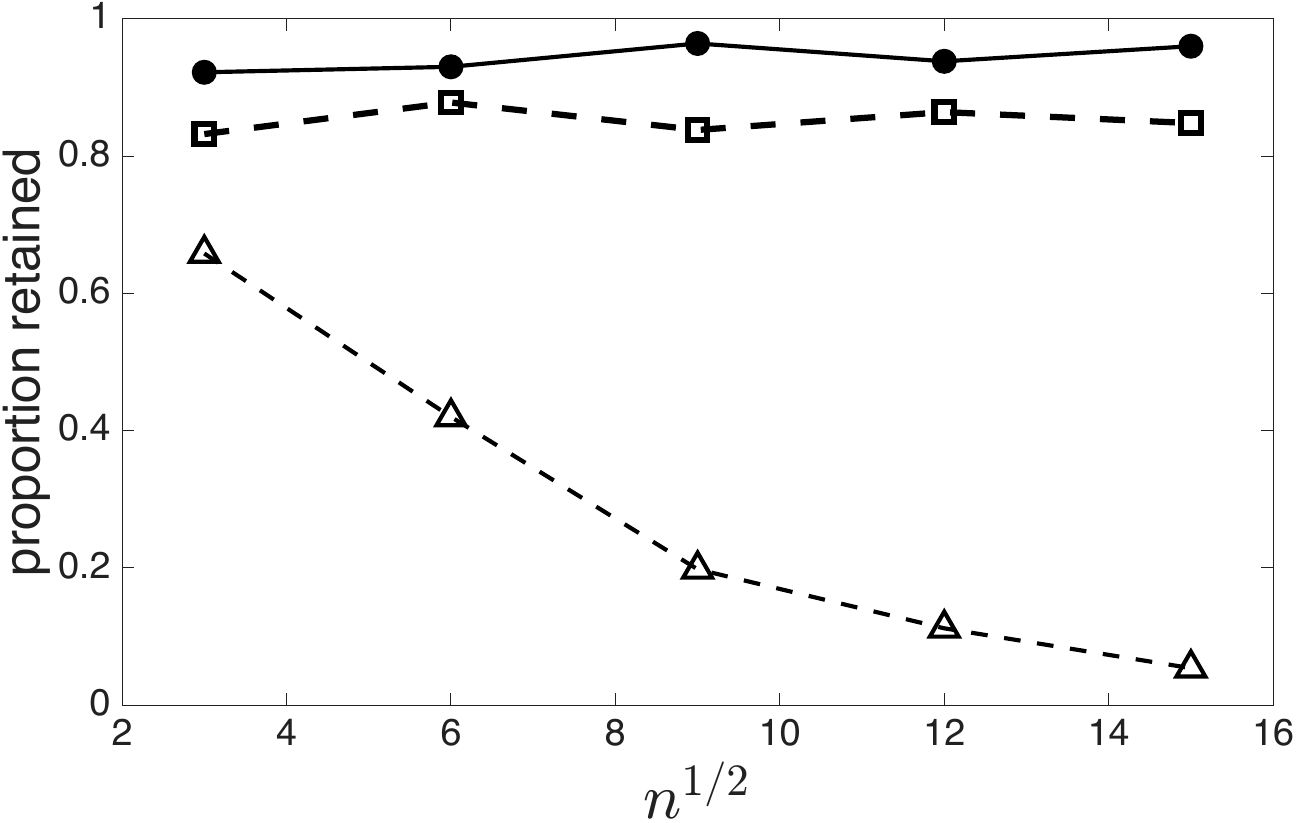} 		
		\includegraphics[trim=0in 0in 0in 0in, clip,width=0.45\linewidth]{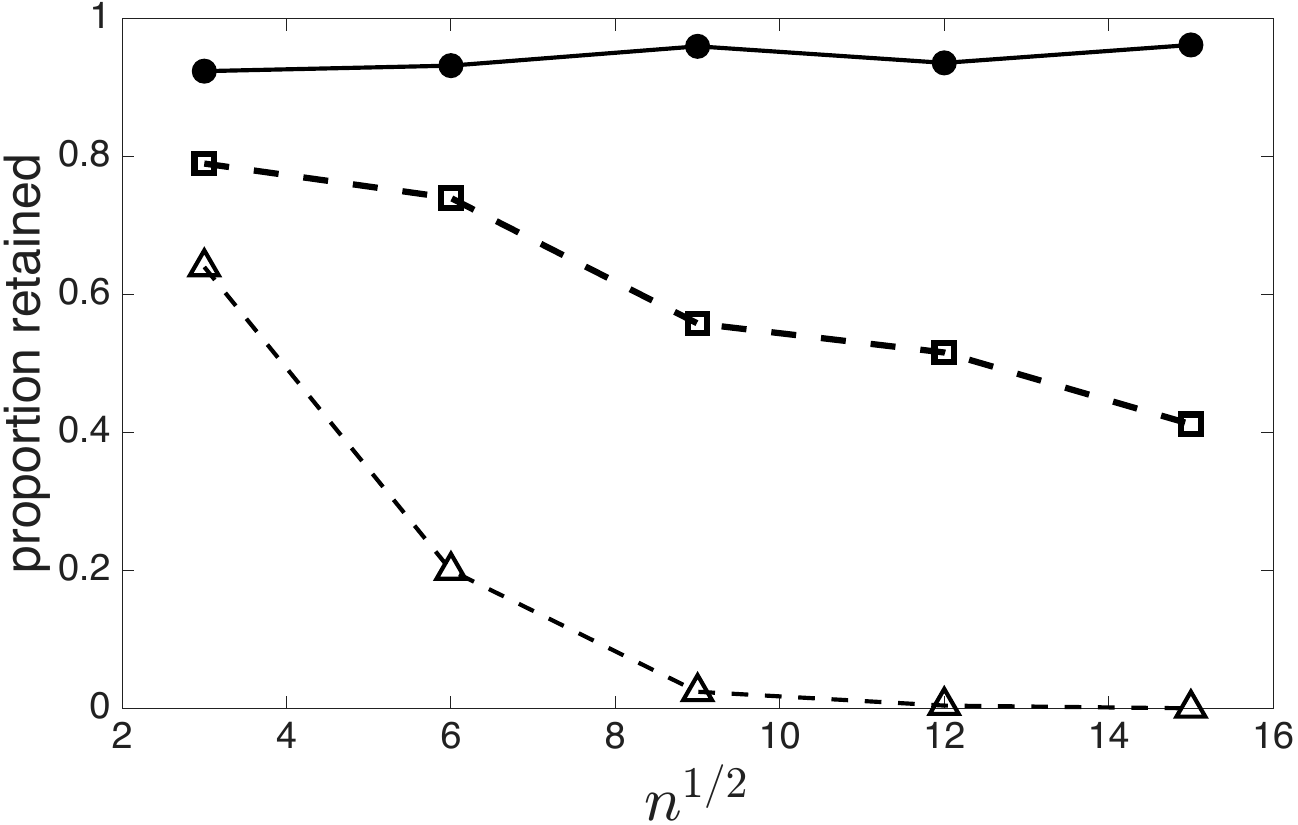} 
		\caption{Simulated retention probabilities, in a nominal 95\% confidence set, for the true model (model 1 in Table \ref{tabVConfig}), a 2-near-equivalent model (model 2) and a 1-near-equivalent-model (model 5) as a function of sample size $n$ for varying signal strengths. Top left to bottom right: $a_0^{(n)} = n^{-1}, n^{-1/2}, n^{-1/4}, n^{-1/6}$. 
			\label{figRates}}
	\end{center}
	
\end{figure}

Figure \ref{figRates} illustrates the permissible rates of decay to distinguish between different models depending on their relationship with the true model in terms of $c$-near-equivalence. Model 1 in Table \ref{tabVConfig} is taken to be the true model $m_0^{(n)}$. For this, the simulated retention probability is close to nominal across all signal strengths. For the 1-near equivalent model (model 2 in Table \ref{tabVConfig}) a signal strength decaying to zero at a rate of $n^{-1/4}$ or slower is needed for the erroneous model to be excluded with high probability at sufficiently large sample size. For the 2-near-equivalent model (model 5 in Table \ref{tabVConfig}), a larger signal strength, decaying at the slow rate $n^{-1/6}$ is needed.

\section{Mixtures of causal models}\label{secMixtures}

In some settings the true causal explanation may well differ between individuals. While ideally the reason for this difference might be understood and the different sets of individuals separated in the analysis, this is not always feasible. On the other hand, the principles of modelling require some stability across individuals \citep[see][for a discussion in a regression setting]{McC2002, McC2022}. Reasonable allowance for different generating mechanisms while respecting the foundations of modelling leads to mixture models in which any individual from a sample is thought of as a draw from a mixture of causal models on account of the label being unknown. The mixture weights correspond to the proportion of individuals for which each causal model is explanatory. A key question is how the confidence sets of models behave in this situation. Intuitively, if the mixture weight on a particular causal model is too small, this lowers the effective signal strength (equivalently sample size) for retaining the model in the confidence set. We thus expect some sharp transitions, with respect to mixture weight, in the ability of the confidence set to cover all true models in the mixture. The ability to detect some models in the mixture will additionally depend on what else is included and its relative geometry. We probe this further, first theoretically for a two-component mixture of Gaussian causal models, and then through simulation for a three-component mixture.

Let $m_{0,1}$ and $m_{0,2}$ be causal models generating probability distributions $F_{0,1}$ and $F_{0,2}$ on the observable outcomes. Then the two-component mixture model is $(1-\theta)F_{0,1} + \theta F_{0,2}$. This should be thought of as specifying that, for an individual $i$ belonging to a group of individuals observed in proportion $(1-\theta)$, the data are generated according to $F_{0,1}$. Assume $F_{0,1}$ and $F_{0,2}$ are mean-zero Gaussian distributions with covariance matrices $\Sigma_{0,1}$ and $\Sigma_{0,2}$ respectively.

Assuming that $\Sigma_{0,1}$ is not equal to $\Sigma_{0,2}$, the two-component mixture model has mean-zero and covariance matrix $\Sigma_{0} = (1-\theta)\Sigma_{0,1} + \theta \Sigma_{0,2}$ but is not Gaussian unless $\theta \in \{0,1\}$. As before, we use the condensed notation $\gamma = \text{vech}(\Sigma) \in \mathbb{R}^{d(d+1)/2}$ and consider notional local models about the point of intersection $\gamma^*\in m_{0,1} \cap m_{0,2} \cap m_1$ where $m_1$ is the postulated model. Let $m_{0,1}^{(n)}, m_{0,2}^{(n)}$ and $m_1^{(n)}$ have elements whose parameters converge to $\gamma^*$ as $n\rightarrow \infty$. By construction, this implies that the parameters associated with all elements of $m_1^{(n)}$ converge to the true parameter $\gamma_0^{(n)}$. A formulation analogous to that of \S \ref{secCalibration} based on contiguous alternatives is constructed by allowing $\theta^{(n)} \rightarrow 0$. Proposition \ref{propMixDist} relates the $c$-equivalence of the mixture and a postulated model $m_1$ to the corresponding $c$-equivalences of the constituent terms in the mixture. This is used in Corollaries \ref{corMix} and \ref{corMix2} to reveal the behaviour of the confidence set in the mixture setting.

\begin{proposition}\label{propMixDist}
	Let $m_{0,1}^{(n)}$ and $m_{0,2}^{(n)}$ be the true models with distributions $F_{0,1}$ and $F_{0,2}$ and parameters $\gamma_{0,1}^{(n)}, \gamma_{0,2}^{(n)}$ in the two-component mixture model $(1-\theta^{(n)})F_{0,1} + \theta^{(n)} F_{0,2}$. Let $\varepsilon:=\theta^{(n)}$ in Definition \ref{defCEquiv}. Then if, for a postulated model $m_1^{(n)}$, the three pairs of models $m_{0,1}^{(n)}$ and $m_{0,2}^{(n)}$, $m_{0,1}^{(n)}$ and $m_1^{(n)}$, and $m_{0,2}^{(n)}$ and $m_1^{(n)}$, are respectively $c$-equivalent, $c_1$-equivalent, and $c_2$-equivalent at $\gamma^*$, then the mixture $m_0^{(n)}$ is $c_m$-equivalent to $m_1^{(n)}$, where 
	\[
	c_m := \max\{\min\{ c, c_1\},\min\{ c, c_2\}\},
	\]
	and the Hausdorff distance satisfies $d_H(m_0^{(n)},m_1^{(n)})=o(\varepsilon^{c_m})$. 
\end{proposition}

From Proposition \ref{propMixDist}, we can apply Proposition \ref{propCSProperties} with the mixture as the true model $m_0$, giving the analogous statement of Corollary \ref{corMix2}. We first characterise the coverage properties for each constituent model in Corollary \ref{corMix}. This amounts to setting $\varepsilon^{c_1}=0$ and $\varepsilon^{c_2}=0$ in the definition of Hausdorff-distance characterising the notion of $c$-equivalence, as if the postulated model $m_1$ is equal to $m_{0,1}$ or $m_{0,2}$, then the notion of $c_1$-equivalence ($c_2$-equivalence respectively) does not make sense.

\begin{corollary}\label{corMix}
	Let the true mixture distribution $(1-\theta)F_{0,1} + \theta F_{0,2}$ have component distributions in models $m_{0,1}^{(n)}$ and $m_{0,2}^{(n)}$ respectively, contiguous at the point of intersection $\gamma^*$. Let $\gamma_{0,1}^{(n)}=\gamma^* + a_{0,1}^{(n)}$ and $\gamma_{0,2}^{(n)}=\gamma^* + a_{0,2}^{(n)}$ be the corresponding parameters. Suppose that $m_{0,1}^{(n)}$ and $m_{0,2}^{(n)}$ are $c$-equivalent at $\gamma^*$ and let the postulated model $m_1$ be either of the two true ones. Let $\mathcal{M}$ be the set of all models not rejected at level $\alpha$ by a likelihood ratio test of $m_1$ against the unconstrained model. Then with $\theta^{(n)} = O(n^{-1/2})$ as $n\rightarrow \infty$:
	\begin{enumerate}
		\item[(i)]  $m_{0,1}^{(n)}$ is retained in $\mathcal{M}$ with probability $1-\alpha$ irrespective of $a_{0,1}^{(n)} \asymp a_{0,2}^{(n)}$;
		\item[(ii)]  $m_{0,2}^{(n)}$ is retained in $\mathcal{M}$ with probability $1-\alpha + o(1)$ if and only if $a_{0,1} \asymp a_{0,2}^{(n)} \lesssim n^{-1/2c}$. Otherwise it is included with probability $o(1)$. 
	\end{enumerate}
	Moreover, if any model in a Markov equivalence class is retained in $\mathcal{M}$, then all models in the same class are retained.
\end{corollary}

Corollary \ref{corMix} states that the $c$-equivalence between the two true models only affects the retention probability for $m_{0,2}$. Figure \ref{figMix} explains the intuition for this: as $\theta^{(n)}\rightarrow 0$, the true mixture model approaches $m_{0,1}$, so $m_{0,1}$ is retained with asymptotically nominal probability regardless of signal strength. In this same limit, the $c$-equivalence between $m_{0,1}$ and $m_{0,2}$ coincides with the $c$-equivalence of $m_{0}$ and $m_{0,2}$.

\begin{figure}
	\begin{center}
		\includegraphics[trim=0in 0in 0in 0in, clip,width=0.45\linewidth]{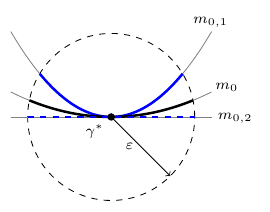}
		\includegraphics[trim=0in 0in 0in 0in, clip,width=0.45\linewidth]{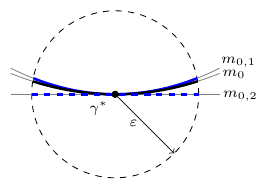}	
		\includegraphics[trim=0in 0in 0in 0in, clip,width=0.3\linewidth]{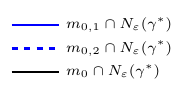} 
		\caption{The mixture model $m_0$ and its constituent models $m_{0,1}$ and $m_{0,2}$, for $\varepsilon= \theta = n^{-1/2}$. Left: small $n$. Right: larger $n$. The right sub-figure is magnified to make the models within the ball visible. 
			\label{figMix}}
	\end{center}
\end{figure}

Corollary \ref{corMix2} covers the case in which the postulated model $m_1$ is not represented by one of the two mixture components. This follows directly from Proposition \ref{propCSProperties} using Proposition \ref{propMixDist}. 

\begin{corollary}\label{corMix2}
	Consider the setting of Corollary \ref{corMix} extended for a postulated model $m_1^{(n)}$ with parameters $\gamma_1^{(n)} = \gamma^* + a_1^{(n)}$, not represented in the true mixture with components from $m_{0,1}^{(n)}$ and $m_{0,2}^{(n)}$. Let $c_1$ and $c_2$ be the $c$-equivalences of $m_1^{(n)}$ and $m_{0,1}^{(n)}$, and $m_1^{(n)}$ and $m_{0,2}^{(n)}$ respectively. Let $\mathcal{M}$ be the set of all models not rejected at level $\alpha$ by a likelihood ratio test of $m_1$ against the unconstrained model. As $n\rightarrow \infty$:
	\begin{enumerate}
		\item[(i)] All models $m_1^{(n)}\sim m_0^{(n)}$ are included in $\mathcal{M}$ with probability $1-\alpha$ irrespective of $a_{0,1}^{(n)} \asymp a_{0,2}^{(n)} \asymp a_1^{(n)}$.
		\item[(ii)]  If $c < \min\{c_1, c_2\}$, then $m_1^{(n)}$ is retained in $\mathcal{M}$ with probability $1-\alpha + o(1)$ if and only if $a_{0,1}^{(n)} \asymp a_{0,2}^{(n)} \asymp a_{1}^{(n)} \lesssim n^{-1/2c}$. Otherwise it is retained with probability $o(1)$. 
		\item[(iii)]  If $\max\{c_1,c_2\} < c$, then $m_1^{(n)}$ is retained in $\mathcal{M}$ with probability $1-\alpha + o(1)$ if and only if $a_{0,1}^{(n)} \asymp a_{0,2}^{(n)} \asymp a_{1}^{(n)} \lesssim n^{-1/2 \max\{c_1, c_2\}}$. Otherwise it is included with probability $o(1)$. 
	\end{enumerate}
	Moreover, if any model in a Markov equivalence class is retained in $\mathcal{M}$, then all models in the same class are retained.
\end{corollary}

Figure \ref{figRatesMix} shows the simulated retention probabilities of a $95\%$ confidence set for $m_{0,1}$ and $m_{0,2}$ when they are are $2-$near-equivalent. This confirms the theoretical results of Corollary \ref{corMix}, $m_{0,1}$ being retained at probability close to $0.95$ regardless of signal strength, and $m_{0,2}$ to being distinguishable from $m_{0,1}$ when the signal strength decays at rate $n^{-1/4}$ or slower. For a postulated model $m_1$ that is $1-$near-equivalent to both $m_{0,1}$ and $m_{0,2}$, the simulated retention probabilities are close to nominal for signal strength decaying faster than $n^{-1/2}$, and otherwise there is power to detect the erroneous model and discard all elements of its equivalence class from the confidence set.

\begin{figure}[h!]
	\begin{center}
		\includegraphics[trim=0in 0in 0in 0in, clip,width=0.45\linewidth]{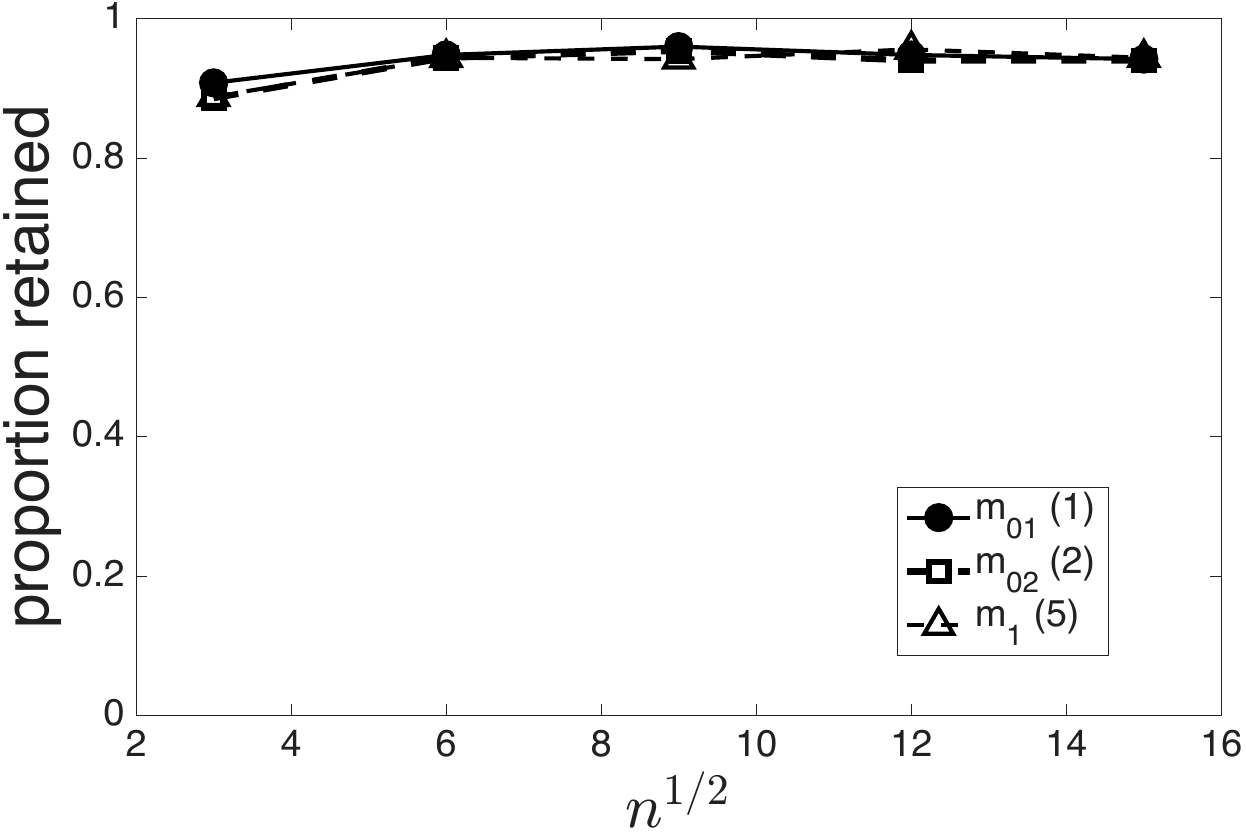}
		\includegraphics[trim=0in 0in 0in 0in, clip,width=0.45\linewidth]{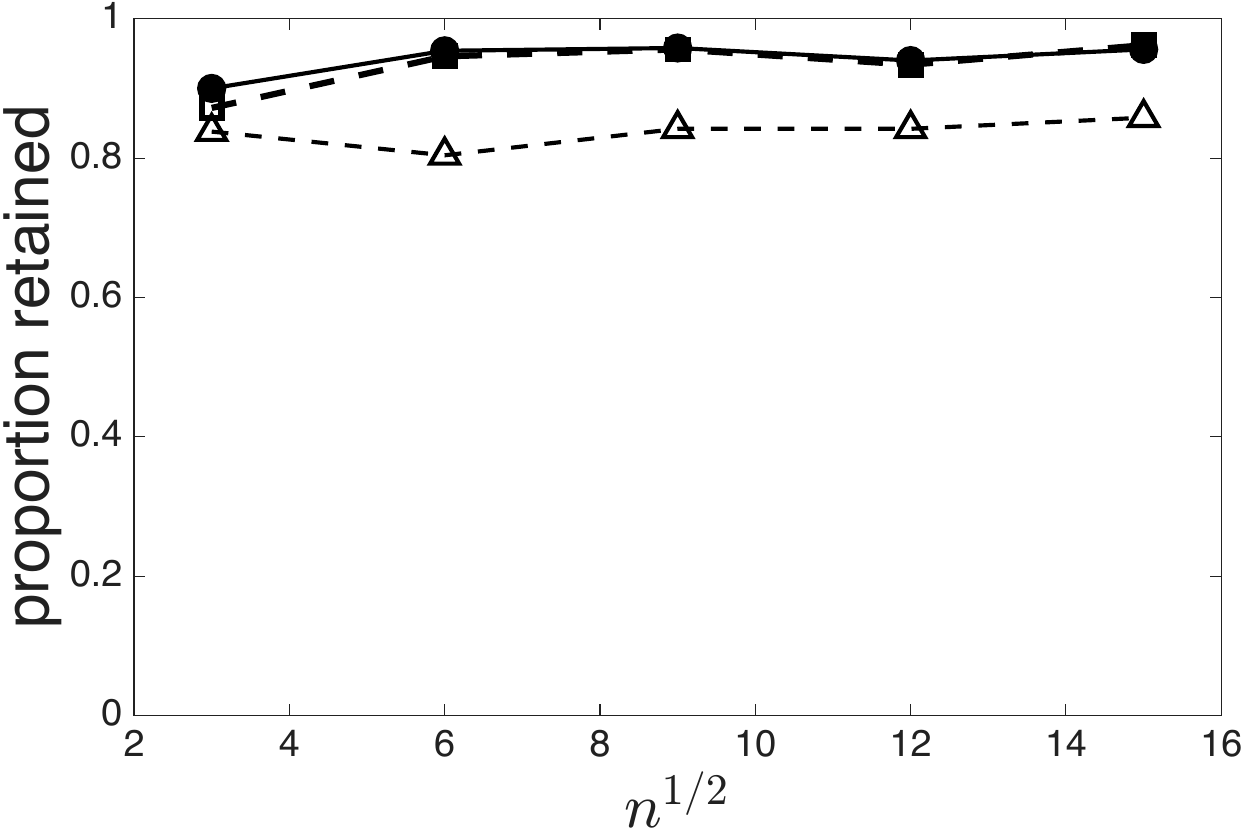} 
		
		\includegraphics[trim=0in 0in 0in 0in, clip,width=0.45\linewidth]{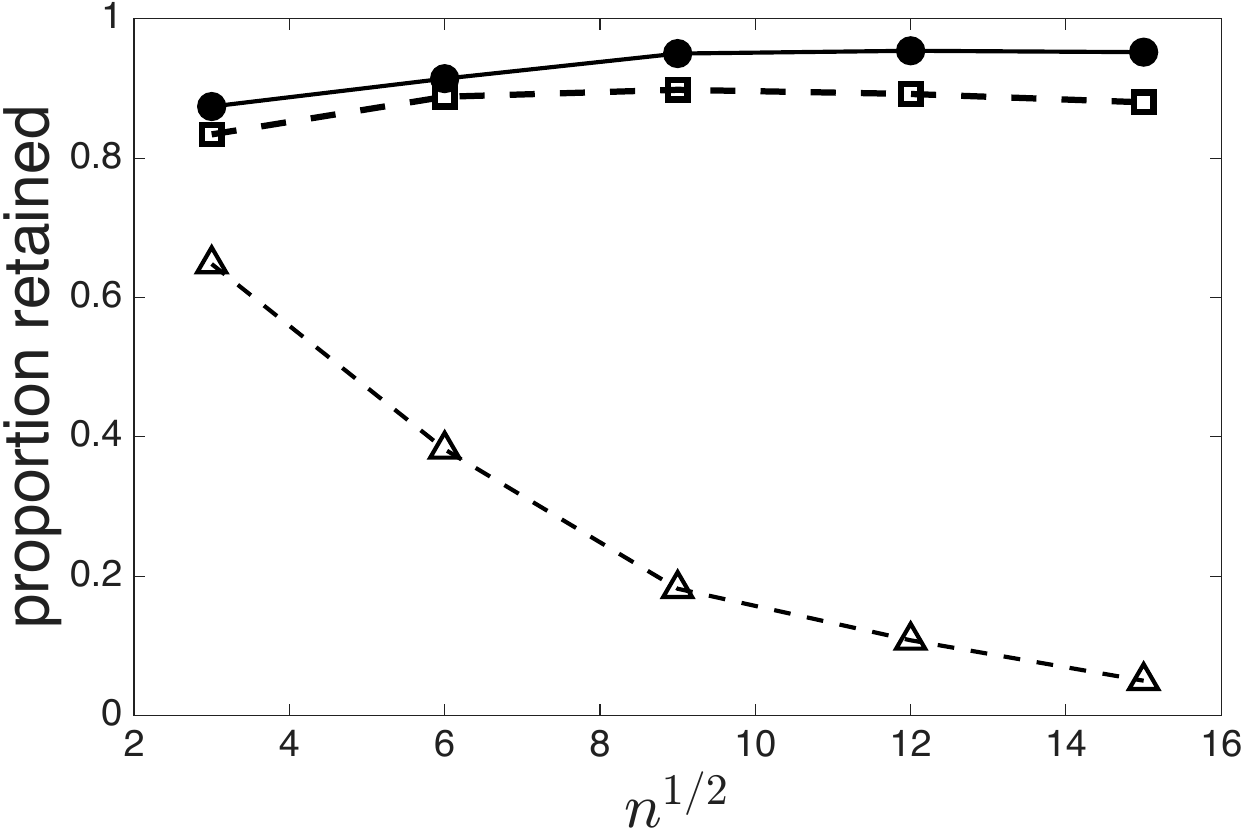} 		
		\includegraphics[trim=0in 0in 0in 0in, clip,width=0.45\linewidth]{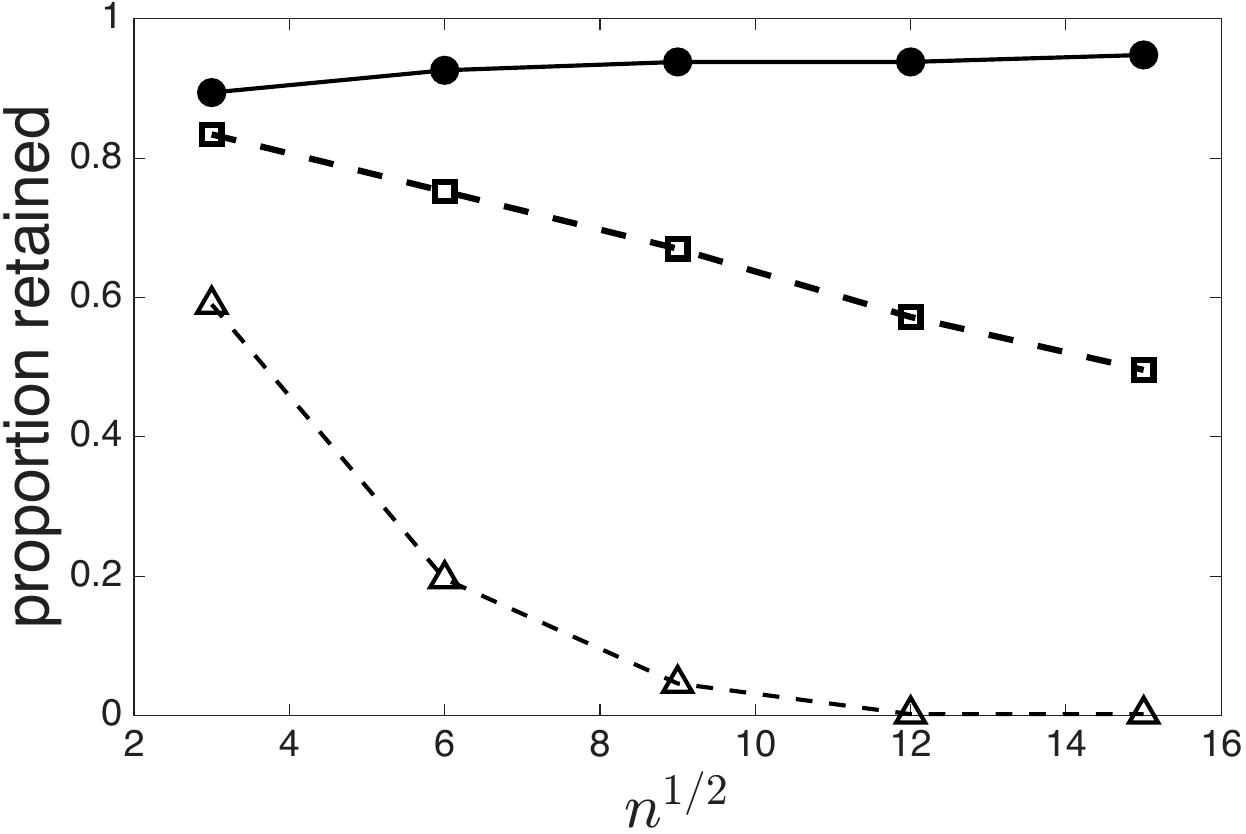} 
		\caption{Simulated retention probabilities for the two 
			true models (models 1 and 2 in Table \ref{tabVConfig}), and a 1-near-equivalent-model (model 5) as we increase sample size $n$ for varying signal strengths. Top left to bottom right: $a_0^{(n)} = n^{-1}, n^{-1/2}, n^{-1/4}, n^{-1/6}$. 
			\label{figRatesMix}}
	\end{center}
\end{figure}

The analogue of Proposition \ref{propKL} for the mixture setting requires relating the behaviour of the Kullback-Leibler-minimising parameter to the quantities governing how the parameter of the mixture approaches the point $\gamma^*$ of intersection. As before, we define $\mathbb{M}_n(\gamma)= \ell(\gamma)/n$ to be the rescaled log-likelihood function and $\mathbb{M}_0^{(n)}(\gamma)=\mathbb{E}\mathbb{M}_n(\gamma)$, depending on $n$ via the expectation operator under the sequence of true parameter values, which is now $\gamma_0^{(n)}= (1-\theta^{(n)})(\gamma^* + a_{0,1}^{(n)}) + \theta^{(n)}(\gamma^* + a_{0,2}^{(n)})$. Proposition \ref{propKLMix2} is a direct extension of Proposition \ref{propKL} and the proof is omitted.

\begin{proposition}\label{propKLMix2}
	Let $\hat{\gamma}_1^{(n)} :=  \arg\max_{\gamma \in \Gamma_1} \mathbb{M}_n(\gamma)$ be the maximum likelihood estimator assuming $m_1$ and let $\gamma_1^0 :=  \arg \max_{\gamma \in \Gamma_1} \mathbb{M}_0^{(n)}(\gamma)$, the Kullback-Leibler minimising projection. Under a sequence of true parameter values of the form $\gamma_0^{(n)}= (1-\theta^{(n)})(\gamma^* + a_{0,1}^{(n)}) + \theta^{(n)}(\gamma^* + a_{0,2}^{(n)})$, $\gamma_1^0 = \gamma^* + O(\max(a_{0,1}^{(n)}, \theta^{(n)}a_{0,2}^{(n)})$ and $\hat{\gamma}_1 = \gamma^* + O_p(\max\{n^{-1/2}, a_{0,1}^{(n)},\theta^{(n)}a_{0,2}^{(n)}\})$.
\end{proposition}

In principle, similar results can be quantified for arbitrarily many mixture components, but since the complexity of the statements is increasing in the number of mixture components, we instead present qualitative insights for three mixture components in Figure \ref{figMix3}.

Suppose that the true distribution belongs to a three-component mixture $\theta_1 F_{0,1} + \theta_2 F_{0,2} + \theta_3 F_{0,3}$ where $(\theta_1,\theta_2,\theta_3)$ belongs to the unit simplex. In Figure \ref{figMix3}, models 1, 2 and 5 from Table \ref{tabVConfig} are the three causal models giving rise to component distributions $F_{0,1}, F_{0,2}$ and $F_{0,3}$. The signal strength for each constituent distribution is $0.25$, the level of the likelihood ratio test is $\alpha = 0.05$ and the sample size is $n = 100$, consistent with the setting of Figure \ref{figCoverage}. Figure \ref{figMix3} depicts, for a grid of values $\theta_1, \theta_2, \theta_3$ in the simplex, the simulated retention probabilities from 100 Monte Carlo replications for each of the models in the true mixture, and for an erroneous postulated model (model 12 from Table \ref{tabVConfig}).

\begin{figure}[h!]
	\centering
	
	\begin{minipage}{0.9\textwidth}
		\centering
		\begin{subfigure}{0.49\textwidth}
			\includegraphics[width=\linewidth, trim=5cm 0cm 5cm 0cm, clip]{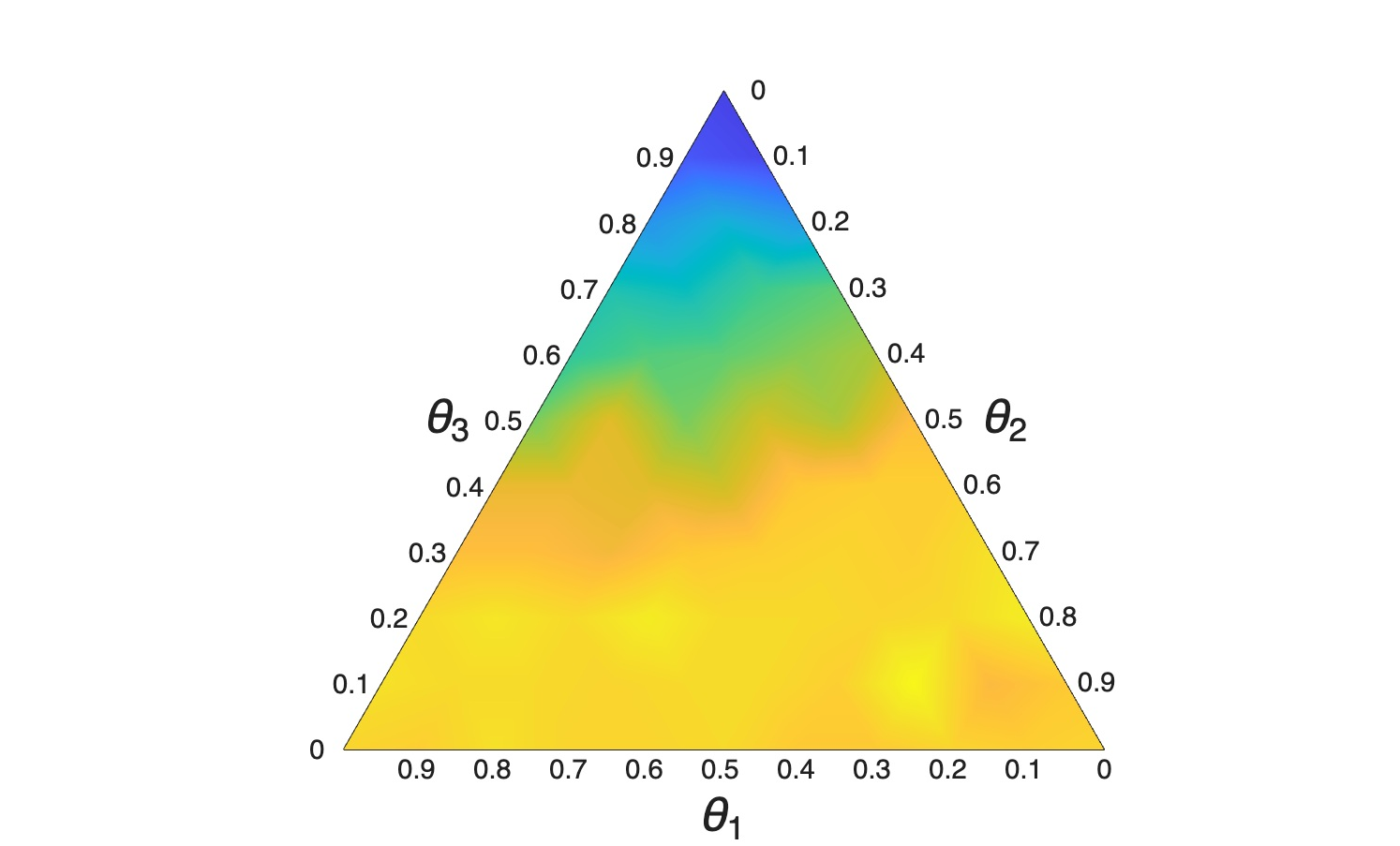}
		\end{subfigure}
		\begin{subfigure}{0.49\textwidth}
			\includegraphics[width=\linewidth, trim=5cm 0cm 5cm 0cm, clip]{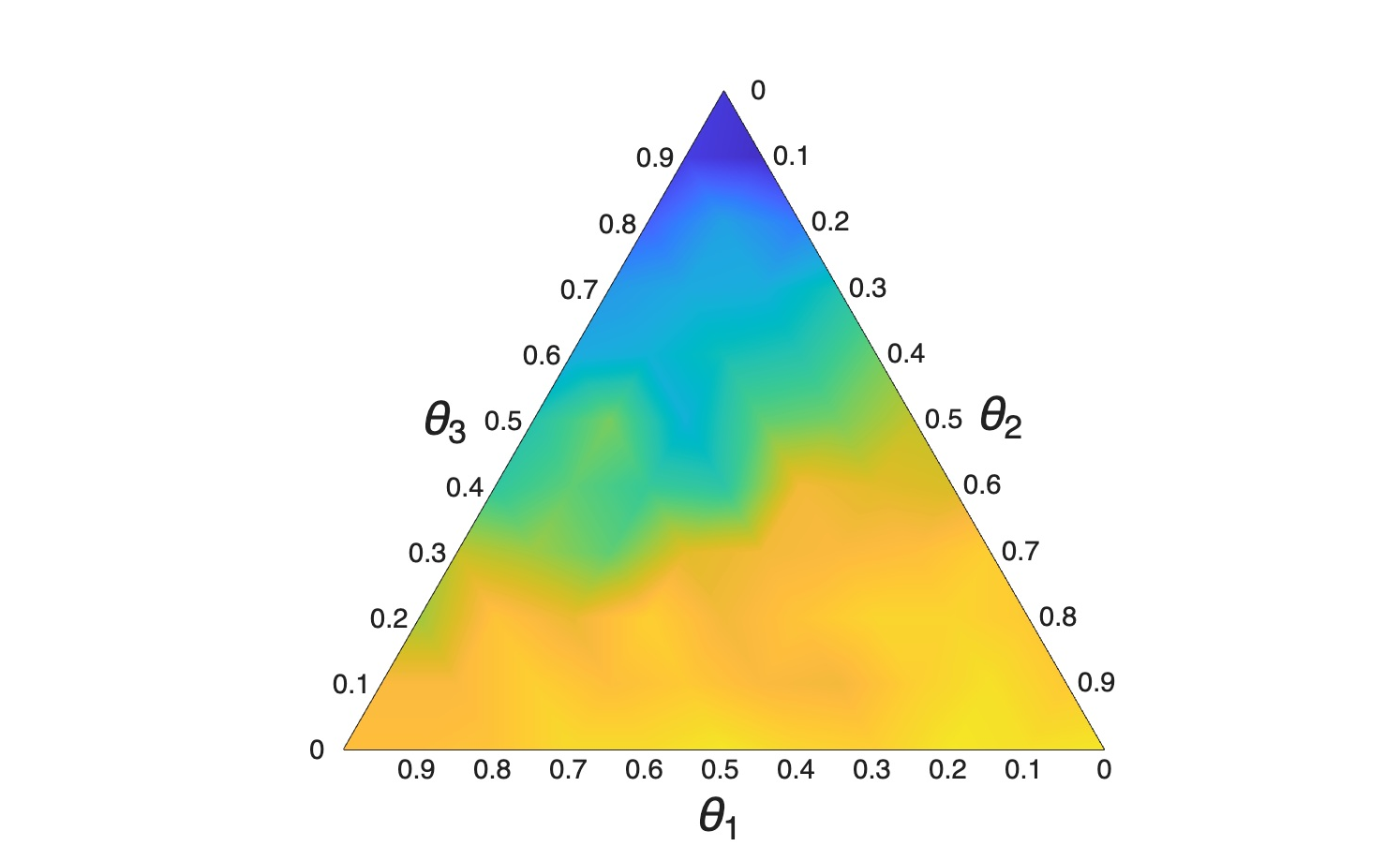}
		\end{subfigure}
		
		\begin{subfigure}{0.49\textwidth}
			\includegraphics[width=\linewidth, trim=5cm 0cm 5cm 0cm, clip]{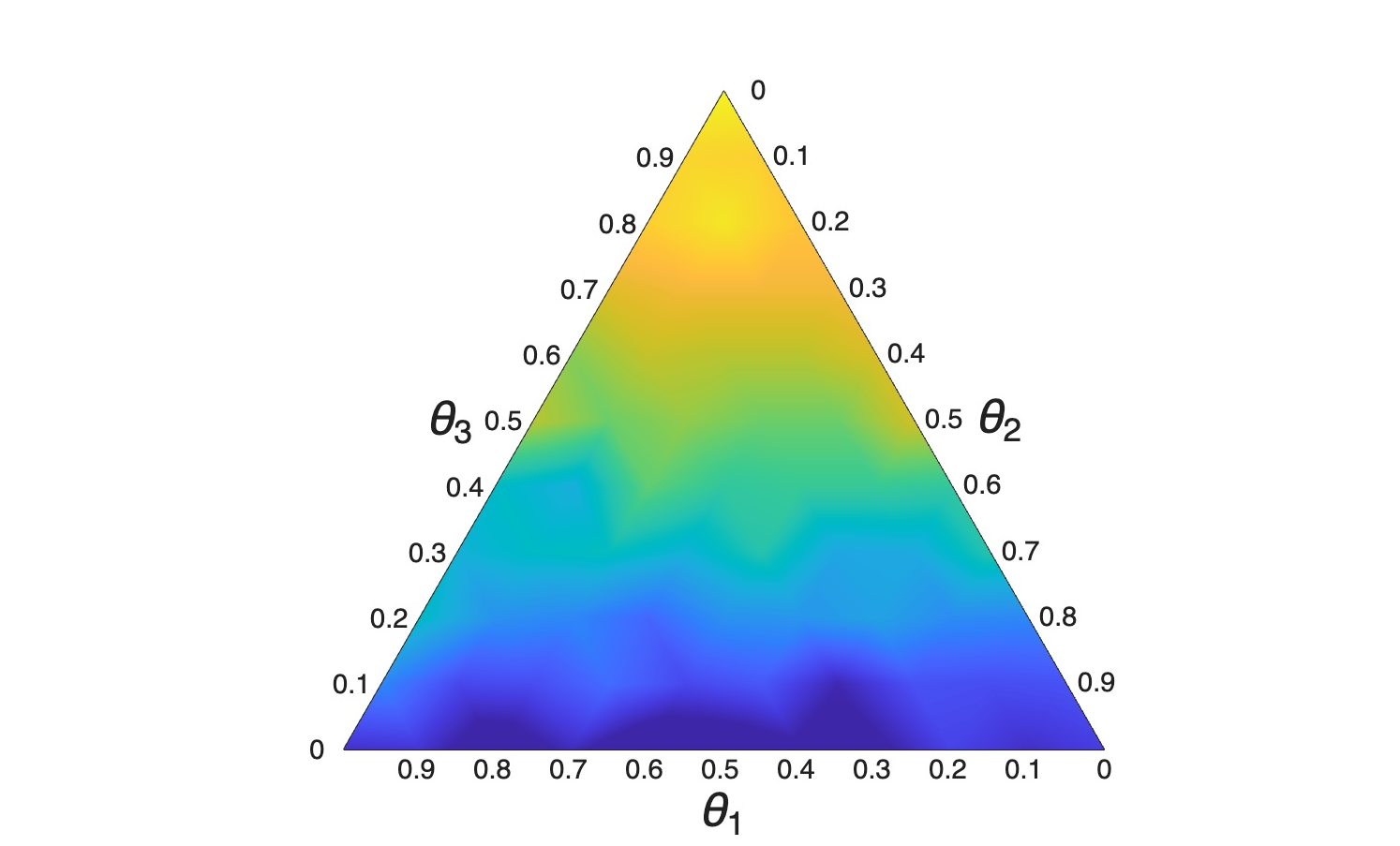}
		\end{subfigure}
		\begin{subfigure}{0.49\textwidth}
			\includegraphics[width=\linewidth, trim=5cm 0cm 5cm 0cm, clip]{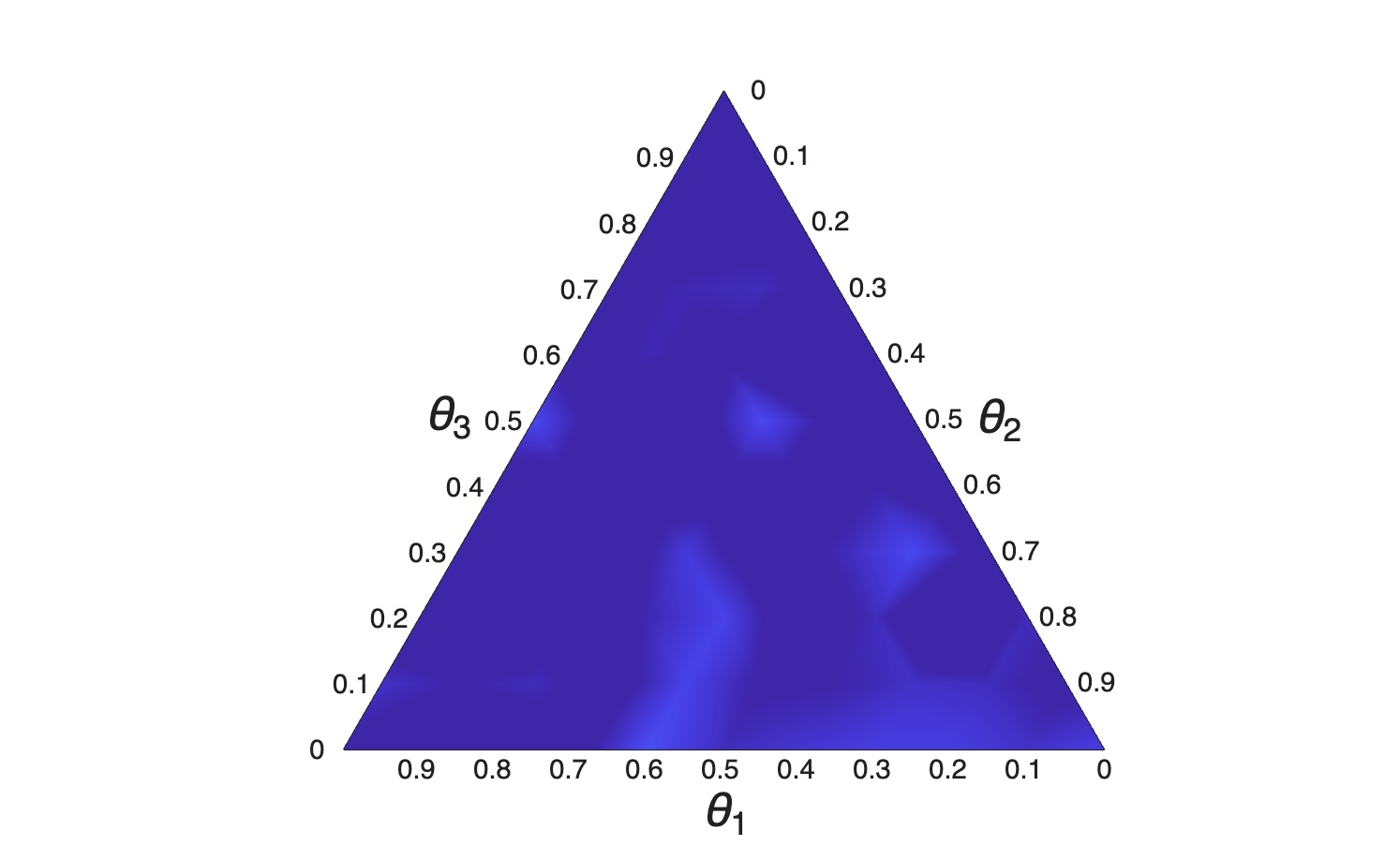}
		\end{subfigure}
	\end{minipage}
	\hfill
	\begin{minipage}{0.06\textwidth}
		\centering
		\vspace{0.5cm} 
		\includegraphics[width=\linewidth]{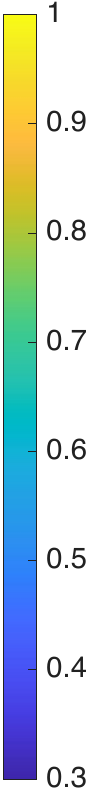}
	\end{minipage}
	\caption{Simulated retention probabilities for model 1 (top left), model 2 (top right), model 5 (bottom left) and model 12 (bottom right) when the true distribution is $\theta_1 F_{0,1} + \theta_2 F_{0,2} + \theta_3 F_{0,3}$ where $F_{0,1}$ belongs to model 1, $F_{0,2}$ to model 2, and $F_{0,3}$ to model 5. \label{figMix3}}
\end{figure}

The simulations confirm intuition: when $\theta_i = 1$, the true mixture distribution is $F_{0,i}$ and from Figure \ref{figCoverage} we expect the corresponding model to be retained with probability close to $1-\alpha$. The top left subfigure of Figure \ref{figMix3} shows the simulated retention probability for model 1, which is close to 0.95 near $\theta_1 = 1$. Analogous conclusions apply near $\theta_2 = 1$ (corresponding to model 2, top right) and $\theta_3 = 1$ (corresponding to model 5, bottom left). 

On account of the 2-near-equivalence of models 1 and 2 at the point of intersection, the two sub-figures on the top row are similar. A true mixture distribution close to a distribution in model 1 or 2, or a mixture of the two, is hard to distinguish from a distribution coming solely from the erroneous 2-near equivalent non-mixture component. The retention probabilities for both model 1 and model 2 are high provided that the contribution from model 5 is sufficiently small ($\theta_3$ below about 0.4). The contrast with the bottom-left sub-figure is due to the $1-$near equivalence of model 5 to the two other components. Only generating distributions that are close to pure mixtures in $F_{0,3}$ (from model 5) lead to high probabilities of retention for model 5, any contamination leading to a distribution that is incompatible with the data at the chosen significance level.

Consider in Figure \ref{figCoverage} the cases where the true model is 1, 2 or 5, which correspond to $\theta_1 = 1$, $\theta_2 = 1$ and $\theta_3 = 1$. These cases align with the vertices of each of the simplex plots. For example, when $\theta_1 = 1$, the true model is retained  with simulated probability close to 0.95, while the corresponding probabilities for models 2 and 5 are close to 0.9 and 0.35 respectively. The probabilities for model 12 (bottom right sub-figure) are close to zero.

\section{Concluding discussion}\label{secDiscuss}

The perspective here is that when scientific or sociological understanding is the goal, as is typically the case in causal-inferential settings, it is in some contexts appropriate to report all causal mechanisms that are compatible with the data. In principle, from a confidence set of causal models, one might seek to refute particular causal relations through a carefully-designed randomised experiment, if practical considerations allow for such.

In the simplest Gaussian settings we have presented a perturbative analysis in which the models contemplated for inclusion in the confidence set are on the borderline of detectability. \citet{Evans2020} showed that his $c$-equivalence classes generalise the notion of Markov equivalence to model pairs whose distinguishability requires the sample size to grow at a particular $c$-dependent rate. The present paper casts these $c$-equivalence classes in terms of confidence sets of models that are statistically indistinguishable from the true generating mechanism at a particular significance level for increasing sample size. We also studied in these terms the ability to identify mixture component models from the composite.

From several perspectives, we have sought to clarify and formalise the implications of \citeauthor{Evans2020}'s (\citeyear{Evans2020}) results, which broadly imply that a stronger signal strength, i.e.~one that decays more slowly with sample size, is needed to exclude models $m_1$ from the confidence set that are locally similar to $m_0$, the model to which the true distribution belongs. This degree of similarity is characterised by a larger $c$ in the $c$-equivalence of $m_1$ to $m_0$.

The formulations of this paper belong to formal theory and are intended to provide understanding rather than direct utility. In particular, the notional asymptotic regimes, as well as the hypothetical replication entailed in the inclusion probabilities for the confidence sets, do not arise in the application of statistical ideas to a single study. One could, however, envisage a use for these insights as a way of combining evidence about the causal mechanism from many different studies in which a confidence set of models was reported. Confidence sets of causal models thus provide new perspectives on, and a possible route to, \citeauthor{Goldthorpe2025}'s (\citeyear{Goldthorpe2025}) open problem of `triangulation' of evidence. How this could be done formally, in a way that would accommodate the different sample sizes, is unclear, but an informal analysis may be suggestive. This could entail, for instance, recording the proportion of times each competing model was found in a confidence set constructed from samples of roughly the same size, as a function of the sample size, and assessing the findings against the theoretical insights of this paper.

One obvious extension of the analysis in \S \ref{secMixtures} is to confidence sets for the mixture models themselves, rather than for the constituent models of the generating mixture model. This is more complicated due to estimation of the nuisance parameter $\theta$ and the many combinations of $c$-equivalences that can arise, both within and between the true and postulated mixture models.

A further extension is to other types of non-Gaussian models. In principle, the construction of \S \ref{secConstruction} can be extended by using the factorisation property of directed acyclic graphs to construct the likelihood ratio statistic for each causal model against a saturated model. In particular applications, successive parametric modelling assumptions for each postulated causal ordering may be convincingly justified, resulting in a regression by composition \citep{RBC}. In the absence of a specific application, however, tightly specified parametric models seem arbitrary, and nonparametric models lead to a less elegant exposition. Discrete and mixed graphs are outside of the scope of some of \citeauthor{Evans2020}'s (\citeyear{Evans2020}) results used in this paper.

\subsection*{Acknowledgements}

We are grateful to Robin Evans for early discussions. EPSRC fellowship EP/T01864X/1 (to HSB) is gratefully acknowledged.

\subsection*{Supplementary material}
The supplementary material contains proofs and additional simulation results.

\newpage

\begin{appendix}

\section{Proofs for Section \ref{secCalibration}}\label{secProofs}

\subsection{Preliminary lemmas}

\begin{lemma}\label{lemmaCS}
	Under the setting of Proposition \ref{propCSProperties}, 
	\begin{enumerate}
		\item[(i)] All models $m_1^{(n)} \sim m_0^{(n)}$ are indistinguishable from $m_0^{(n)}$ in a chi-square assessment irrespective of $a_0^{(n)} \asymp a_1^{(n)}$. 
		\item[(ii)] All models $m_1^{(n)}$ that are c-equivalent to  $m_0^{(n)}$ at $\gamma^*$ are indistinguishable from $m_0^{(n)}$ in a chi-square assessment if $a_0^{(n)} = O(n^{-1/2c})$. 
	\end{enumerate}
\end{lemma}

\begin{proof}
	We restate for ease of reference some results from \citet[][]{vdV2000}. Introduce parameter spaces $\Gamma_0$ and $\Gamma_1$ such that $\gamma_0 \in \Gamma_0$ corresponds to the true model $m_0$ and $\gamma_1 \in \Gamma_1$ corresponds to the permissible parameters under the postulated model $m_1$. By the statement of Proposition \ref{propCSProperties}, $\gamma_0$ belongs to the interior of its parameter space. Let $\gamma \in \Gamma$ be the parameter under the unconstrained model. In a notation mirroring that of \citet[][]{vdV2000}, introduce local parameter spaces $H_n = \sqrt{n}(\Gamma - \gamma_0)$ and $H_{n,0} = \sqrt{n}(\Gamma_0 - \gamma_0)$, and $H_{n,1} = \sqrt{n}(\Gamma_1 - \gamma_0)$. In this parameterisation the null hypothesis $\gamma = \gamma_0 + h_j/\sqrt{n} \in \Gamma_j$ corresponds to $h_j \in H_{n,j}$ for $j \in \{0,1\}$. This inclusion reflects a true null hypothesis if $j=0$ and an erroneous one if $j=1$. The unconstrained alternative for the likelihood-ratio assessment, $\gamma \in \Gamma$, corresponds to $h \in H_n$ in the local parametrisation. Suppose that the sets $H_{n,0}, H_{n,1}$ and $H_n$ converge to sets $H_0, H_1$ and $H$ respectively, then by \citet[][Theorem 16.7]{vdV2000}, the likelihood ratio statistic $\Lambda_n$ for testing 
	$\gamma \in \Gamma_0$, converges in distribution under $\gamma_0 + h_0/\sqrt{n}$ to $\Lambda$, where 
	\begin{align}\label{eqLambda}
	\begin{split}
	\Lambda &= \inf_{h_0\in H_0}(X-h_0)^\T I_{\gamma_0}(X-h_0) - \inf_{h_0\in H}(X-h_0)^\T I_{\gamma_0}(X-h_0) \\ 
	&=: \| I_{\gamma_0}^{1/2}X - I_{\gamma_0}^{1/2}H_0\|^2 - \|I_{\gamma_0}^{1/2}X - I_{\gamma_0}^{1/2}H\|^2, 
	\end{split}
	\end{align}
	for $X$ normally distributed with mean $h_0$ and covariance matrix $I_{\gamma_0}^{-1}$. The second line is a condensed notation for the first, due to \citet{vdV2000}. With $\gamma_0\in \text{int}(\Gamma_0)$, $H$ is, under this contiguous framework, the full space $\mathbb{R}^{d(d+1)/2}$, thus the second term on the right hand side of \eqref{eqLambda} is zero. Expressing $\Lambda$ as 
	\begin{align*}
	\Lambda &=  \| Z + I_{\gamma_0}^{1/2}h_0 - H_0'\|^2, 
	\end{align*}
	where $Z$ has a standard Normal distribution and $I_{\gamma_0}^{1/2}h_0 \in H_0' := I_{\gamma_0}^{1/2}H_0$, it follows by \citet[e.g.][Lemma 16.6]{vdV2000} that the limit distribution of $\Lambda_n$ under $\gamma_0 + h_0/\sqrt{n}$ is $\chi^2$ with $\nu$ degrees of freedom, where $\nu$ is the number of constraints in $H_0$ relative to $H$ .
	
	The likelihood ratio statistic $\Lambda_n$ for testing the postulated model $m_1$ under which $\gamma = \gamma_0 + h_1/\sqrt{n} \in \Gamma_1$ 
	converges in distribution under $\gamma_0 + h_0/\sqrt{n}$  to $\Lambda$, where 
	\begin{align*}
	\Lambda &= \inf_{h_1\in H_1}(X-h_1)^\T I_{\gamma_1^0}(X-h_1) - \inf_{h_1\in H}(X-h_1)^\T I_{\gamma_1^0}(X-h_1) \\ 
	&=: \| I_{\gamma_1^0}^{1/2}X - I_{\gamma_1^0}^{1/2}H_1\|^2 - \|I_{\gamma_1^0}^{1/2}X - I_{\gamma_1^0}^{1/2}H\|^2, \label{Lambda2}
	\end{align*}
	with $X$ as defined previously. It is clear in the condensed notation that the second term is that of \eqref{eqLambda}. Thus
	\begin{align}
	\Lambda = \|Z + I_{\gamma_1^0}^{1/2}h_0 - I_{\gamma_1^0}^{1/2}H_1\|^2
	\end{align}
	for a standard normal vector $Z$. We retain the notation $I_{\gamma_1^0}$ indicating that the information matrix is evaluated at the parameter under $m_1$, even though this is $I_{\gamma_0}$ in the relevant limit. The distribution of $\Lambda$ is non-central $\chi^2$ with non-centrality parameter $\delta = \|I_{\gamma_1^0}^{1/2}h_0 - I_{\gamma_1^0}^{1/2}H_1\|$, and degrees of freedom equal to the number of constraints in $H_1$ relative to $H$. As long as $m_0 \not\sim m_1$, so that $h_0 \notin H_1$, $\delta$ is positive. 
	
	Under the setting of Proposition \ref{propCSProperties}, the permissible parameters under both models approach the intersection $\gamma^*$: $\gamma_0^{(n)} = \gamma^* + a_0^{(n)}$ and $\gamma_1^{(n)} = \gamma^* + a_1^{(n)}$. Under Markov equivalence $m_1^{(n)} \sim m_0^{(n)}$, $\hat{\gamma}_1^{(n)} = \hat{\gamma}_0^{(n)}$ and $\gamma_1^{(n)} = \gamma_0^{(n)}$. Thus the log-likelihood functions under the two model constraints, and the distribution of the associated likelihood-ratio statistic are the same, and the models are indistinguishable. 
	
	The constraint $\gamma_0^{(n)} \in \Gamma_0$ corresponds to $h_0 \in H_{n,0}$ as before, and we recover the $\chi^2$ distribution with $\nu$ degrees of freedom under a true postulated model. For an erroneous postulated model, the relevant quantity $\gamma_1^{(n)} = \gamma_0^{(n)} + h_1/\sqrt{n} \in \Gamma_1$, is, in terms of $\gamma^*$,
	$\gamma_1^{(n)} = \gamma^* + a_0^{(n)} + h_1/\sqrt{n}$, so that  $h_1 = \sqrt{n}(a_1^{(n)} - a_0^{(n)})$, whose limit is $k$ in the notation of \citet[][]{Evans2020}. Then under the true $\gamma_0^{(n)}$, the non-centrality parameter $\delta = \|I_{\gamma_1^0}^{1/2}h_0 - I_{\gamma_1^0}^{1/2}H_1\| = \inf_{h_1} (h_1^T I_{\gamma_1^0} h_1)^{1/2}$. 
	
	Suppose that $m_1^{(n)}$ is $c$-equivalent to $m_0^{(n)}$, then, by definition, $a_1^{(n)} = a_0^{(n)} +  o(\|a_0^{(n)}\|^c)$ \citep[see][p.~3538]{Evans2020}. If $a_0^{(n)} \lesssim n^{-1/2c}$ then $a_1^{(n)} = o(n^{-1/2})$ and $h_1 = \sqrt{n}(a_1^{(n)} - a_0^{(n)}) = o(1)$ 
	so that $\delta \rightarrow 0$ and $\Lambda$ converges in distribution to a central $\chi^2$ distribution. Thus, $m_1^{(n)}$ is indistinguishable from $m_0^{(n)}$ in a $\chi^2$ assessment.

\end{proof}

For notational convenience here and henceforth, we will not always be explicit about the dependence on $n$ induced through $\gamma_0^{(n)}=\gamma^* + a_0^{(n)}$. We thus write $\gamma_0$ in place of $\gamma_0^{(n)}$ and $\gamma_1^0:=\argmax_{\gamma \in \Gamma_1} \MM_0^{(n)}$ where $\MM_0^{(n)}:=\EE_{\gamma_0}\MM_n(\gamma)$.

\begin{lemma}\label{lemmavdVW}
	There exists $C<\infty$ and $\delta>0$ such that, for all $\gamma$ satisfying $d(\gamma,\gamma_1^0)<\delta$,
	\[
	\MM_0^{(n)}(\gamma) -  \MM_0^{(n)}(\gamma_1^0) \leq -C d^2(\gamma,\gamma_1^0).
	\]
\end{lemma}

\begin{proof}
	This is by Taylor expansion using the definition of $\gamma_1^0$ as a maximiser of $\MM_0^{(n)}$, implying negative definiteness of the Hessian.
\end{proof}

\begin{lemma}\label{lemmaGaussianLik}
	Under a mean-zero Gaussian model of dimension $d$,
	\begin{align*}
	\MM_0^{(n)}(\gamma) = &\;  -\vfrac{d}{2}\log(2\pi) - \vhalf\log (\det \Sigma_{\gamma}) - \vhalf \tr(\Sigma_\gamma^{-1}\Sigma_{\gamma_0}) \\
	\MM_0(\gamma) = &\;  -\vfrac{d}{2}\log(2\pi) - \vhalf\log (\det \Sigma_{\gamma}) - \vhalf \tr(\Sigma_\gamma^{-1}\Sigma_{\gamma^*})
	\end{align*}
	where $\MM_0(\gamma)=\lim_n 	\MM_0^{(n)}(\gamma)$ and $\Sigma_\gamma = \vech^{-1}(\gamma)$.
\end{lemma}

\begin{proof}
	The form for $\MM_0^{(n)}(\gamma)$ can be checked by direct calculation. The form for $\MM_0(\gamma)$ follows on taking the appropriate limit.
\end{proof}

\subsection{Proofs of main results}

\begin{proof}[Proof of Proposition \ref{propCSProperties}]
	
	The construction of the confidence set uses the chi-square assessment $\Lambda_n \leq \chi^2_{\nu,1-\alpha}$, where $\chi^2_{\nu,1-\alpha}$ is the ($1-\alpha$)-quantile of the $\chi^2_\nu$ distribution. The true model $m_0^{(n)}$ is thus retained with asymptotic probability $1-\alpha$ by standard arguments, recalled in the proof of Lemma \ref{lemmaCS}. 
	
	Suppose that $m_1^{(n)} \sim m_0^{(n)}$, then by Lemma \ref{lemmaCS}, $m_1^{(n)}$ is indistinguishable from $m_0^{(n)}$ irrespective of $a_0^{(n)}$ in a chi-square assessment. Therefore, $m_1^{(n)}$ is also retained in the confidence set with probability $1-\alpha$. Suppose that $m_1^{(n)}$ is $c$-equivalent to $m_0^{(n)}$ at $\gamma^*$ and $a_0^{(n)} = O(n^{-1/2c})$. Then by Lemma \ref{lemmaCS}, $\pr(\Lambda_n \leq \chi^2_{\nu,1-\alpha}) = 1-\alpha + o(1)$ as $n\rightarrow \infty$. Otherwise, $m_1^{(n)}$ can be distinguished from $m_0^{(n)}$, and $\pr(\Lambda_n \leq \chi^2_{\nu,1-\alpha}) = o(1)$ as $n\rightarrow \infty$. 		
\end{proof}

\begin{proof}[Proof of Proposition \ref{propKL}]		
	The first part of the proof is based on that of \citet[][Theorem 3.2.5]{vdVW1996}. Rather than directly applying their theorem, which requires verifying a condition on the expected modulus of continuity of a particular empirical process, we found it easier to to use an explicit calculation for an object that appears earlier in their proof strategy. We therefore provide a version of the full argument here for ease of reference, with some clarifications. 
	
	Let $\gamma_1^0$ be as defined in Lemma \ref{lemmavdVW}. For $n,j \in \mathbb{N}$, define the sequence of sets
	\[
	S_{jn} := \left\{\gamma \in \Gamma_1: 2^{j-1} < r_{n}d(\gamma,\gamma_1^0) \leq 2^{j}\right\}, \quad j \in \mathbb{N},
	\]
	where $r_{n}$ is an divergent sequence of positive real numbers and $d(\gamma,\gamma_1^0)$ is the maximum elementwise distance between $\gamma$ and $\gamma_1^0$. Since we are treating the dimension of $\gamma$ as fixed, it is, however, immaterial which distance is used.

	For $j_0\in \NN$ and $\eta>0$ we have the decomposition into disjoint sets
	\begin{align*}
	\{\gamma: r_n d(\gamma, \gamma_1^0)>2^{j_0}\} \; \subseteq &\; \mathlarger{\cup}_{j\geq j_0}\{\gamma: \gamma \in S_{jn}\} \\
	= & \; \bigl\{\mathlarger{\cup}_{j\geq j_0: 2^j \leq \eta r_n}\{\gamma: \gamma \in S_{jn}\}\bigr\} \; \mathlarger{\cup} \; \bigl\{ \mathlarger{\cup}_{j\geq j_0: 2^j > \eta r_n}\{\gamma: \gamma \in S_{jn}\}\bigr\}. 
	\end{align*}
	Thus, since $2^j>\eta r_n$ implies the containments
	\[
	S_{jn} \subseteq \{\gamma: r_n d(\gamma,\gamma_1^0)>2^{j-1}\} \subseteq \{\gamma:  d(\gamma,\gamma_1^0)>\vhalf \eta \},
	\]
	we have
	\begin{align*}
	\pr(r_{n}d(\hat\gamma_1,\gamma_1^0) > 2^{j_0})\; \leq & \; \sum_{j\geq j_0: 2^j \leq \eta r_n}\hspace{-0.2cm}\pr(\hat{\gamma}_1 \in S_{jn}) + \pr(2\hspace{0.5pt}d(\hat\gamma_1,\gamma_1^0)> \eta) \\
	\leq & \; \sum_{j\geq j_0: 2^j \leq \eta r_n}\hspace{-0.2cm}\pr\Bigl(\sup_{\,\gamma\in S_{jn}}(\MM_n(\gamma) - \MM_{n}(\gamma_1^0))\geq 0\Bigr) + \pr(2\hspace{0.5pt}d(\hat\gamma_1,\gamma_1^0)> \eta).
	\end{align*}
	The second term tends to zero by consistency of $\hat\gamma_1$ to $\gamma_1^0$, which follows from e.g.~\citet[][Theorem 5.7]{vdV2000} using a standard identifiability argument and the fact that
	\[
	\sup_{\gamma\in \Gamma_1}|\MM_n(\gamma) - \MM_0^{(n)}(\gamma)|=\vhalf \sup_{\gamma\in \Gamma_1}|\tr\{\Sigma_\gamma^{-1}(\Sigma_{\gamma_0}-S_n)\}|\rightarrow_p 0
	\]
	by Lemma \ref{lemmaGaussianLik}.
	
	By Lemma \ref{lemmavdVW}, there exists a $C<\infty$, $\delta>0$ and $n_0\in \NN$ such that for all $n>n_0$ and for all $\gamma\in S_{jn}$ such that $\MM_n(\gamma) - \MM_{n}(\gamma_1^0))\geq 0$,
	\begin{align*}
	\mathbb{M}_{n}(\gamma) -  \mathbb{M}_{n}(\gamma_1^0) \; + \; &   \mathbb{M}^{(n)}(\gamma_1^0) -  \mathbb{M}_0^{(n)}(\gamma) \\
	\geq \; & \; C d^{2}(\gamma,\gamma_1^0) \\
	> \; & \; C(r_{n}^{-1}2^{j-1})^{2} \\
	= \; & \; C r_{n}^{-2}2^{2j-2}.
	\end{align*}
	It follows, on letting $\WW(\gamma)=\MM_n(\gamma)-\MM_0^{(n)}(\gamma)$, that
	\begin{equation}\label{eqProbM}
	\pr\Bigl(\sup_{\,\gamma\in S_{jn}}(\MM_n(\gamma) - \MM_{n}(\gamma_1^0))\geq 0\Bigr) \leq \pr\Bigl(\sup_{\,\gamma\in S_{jn}}(\WW_n(\gamma) - \WW_{n}(\gamma_1^0))>  C r_{n}^{-2}2^{2j-2}\Bigr).
	\end{equation}
	On reparametrising in terms of $\Sigma_\gamma=\vech^{-1}(\gamma)$,
	\begin{align*}
	\WW_n(\gamma) - \WW_{n}(\gamma_1^0) = \;& \; \vhalf \tr\{\Sigma_\gamma^{-1}(\Sigma_{\gamma_0}-S_n)\} - \vhalf \tr\{\Sigma_{\gamma_1^0}^{-1}(\Sigma_{\gamma_0}-S_n)\} \\
	= \;& \; \vhalf \tr\{(\Sigma_\gamma^{-1} - \Sigma_{\gamma_1^0}^{-1})(\Sigma_{\gamma_0}-S_n)\} \\ = \;& \; \vhalf \sum_{i,j}(\Sigma_\gamma^{-1} - \Sigma_{\gamma_1^0}^{-1})_{ij}(\Sigma_{\gamma_0}-S_n)_{ji}
	\end{align*}
	by Lemma \ref{lemmaGaussianLik}. Let $(\xi_{k}^\gamma, \lambda_k^\gamma)$ and $(\xi_{k}^{\gamma_1^0}, \lambda_k^{\gamma_1^0})$ denote the $k$th ordered eigenvector and eigenvalue pair for $\Sigma_\gamma$ and $\Sigma_{\gamma_1^0}$ respectively. By the complex variables argument presented by \citet[][p.608]{Battey2019},
	\[
	\sum_{i,j}(\Sigma_\gamma^{-1} - \Sigma_{\gamma_1^0}^{-1})_{ij}(\Sigma_{\gamma_0}-S_n)_{ji} = -\sum_{k,v,r,s}(\lambda_k^{\gamma}\lambda_{v}^{\gamma_1^0})^{-1} \xi_{rk}^\gamma \xi_{sv}^{\gamma_1^0}(\Sigma_{\gamma}-\Sigma_{\gamma_1^0})_{rs}\sum_{i,j}\xi_{ik}^{\gamma}\xi_{jv}^{\gamma_1^0}(\Sigma_{\gamma_0}-S_n)_{ji},
	\]
	and since 
	\[
	\frac{2^{j-1}}{r_n}<(\Sigma_{\gamma}-\Sigma_{\gamma_1^0})_{rs}<\frac{2^j}{r_n}, \quad \quad \gamma\in S_{jn},
	\]
	Markov's inequality applied to \eqref{eqProbM} gives
	\[
	\pr(r_n d(\hat \gamma_1, \gamma_1^0)>2^{j_0}) \leq C \hspace{-0.2cm} \sum_{j\geq j^{0}: 2^{j} \leq \eta r_n} \frac{r_n^2 n^{-1/2}}{r_n 2^{2j - 2}} = C r_n n^{-1/2}  \sum_{j\geq j^{0}} \frac{1}{2^{2j - 2}}.
	\]
	The sum on the right hand side is uniformly bounded in $r_n$, $j_0$ and $\eta$. It follows that the sharpest rate $r_n$ such that the right hand side is $O(1)$ is $r_n=\sqrt{n}$.
	
	The above calculation establishes that $\hat\gamma_1=\gamma_1^0+O_p(n^{1/2})$; to complete the proof of Proposition \ref{propKL} we need additionally to show that $\gamma_1^0=\gamma^* + O(a_0^{(n)})$, where the point of intersection satisfies
	\[
	\gamma^*=\argmax_{\gamma\in \Gamma_1}\MM_0(\gamma), \quad \MM_0(\gamma):=\lim_{n\rightarrow \infty} \MM_0^{(n)}(\gamma).
	\]
	With $A_0^{(n)}:=\vech(a_0^{(n)})$,
	\[
	\MM_0^{(n)}(\gamma)-\MM_0(\gamma)= \vhalf \tr\{\Sigma_\gamma^{-1}(\Sigma_{\gamma^*} - \Sigma_{\gamma_0})\} = \vhalf \tr\{\Sigma_\gamma^{-1}A_0^{(n)}\}
	\]
	by Lemma \ref{lemmaGaussianLik}, from which the conclusion follows.
	
\end{proof}

\section{Proofs for Section \ref{secMixtures}}\label{secProofsMix}

\begin{lemma}\label{lemmaHausDistM01}
	Suppose we have a two-component mixture model $m_0$ defined as $(1-\theta)F_{0,1} + \theta F_{0,2}$ whose constituent distributions belong to $m_{0,1}$ and $m_{0,2}$. Then, $d_H(m_0, m_{0,1}) = \theta d_H(m_{0,1}, m_{0,2})$ and $d_H(m_0, m_{0,2}) = (1-\theta)d_H(m_{0,1}, m_{0,2})$. An analogous statement applies when the models are intersected with $N_{\varepsilon}(\gamma^*)$, as in Definition \ref{defCEquiv}.
\end{lemma}

\begin{proof} The Hausdorff distance between $m_0$ and $m_{0,1}$ is
	\begin{align*}
	d_H(m_0,m_{0,1}) = & \,  \max\biggl\{\sup_{\gamma_0\in m_0} \inf_{\gamma_{0,1} \in m_{0,1}}\|\gamma_0-\gamma_{0,1}\|_{2}, \sup_{\gamma_{0,1}\in m_{0,1}} \inf_{\gamma_0\in m_0}\|\gamma_0-\gamma_{0,1}\|_{2}\biggr\} \\
	& \hspace{1.7cm} \sup_{\gamma_{0,1}\in m_{0,1}} \inf_{\gamma_{0,1}\in m_{0,1} \gamma_{0,2}\in m_{0,2}} \|(1-\theta)\gamma_{0,1} + \theta \gamma_{0,2}-\gamma_{0,1}\|_{2}\biggr\} \\ 
	= &\,   \max\biggl\{\sup_{\gamma_{0,1}\in m_{0,1} \gamma_{0,2}\in m_{0,2}} \inf_{\gamma_{0,1} \in m_{0,1}}\|\theta( \gamma_{0,2}-\gamma_{0,1})\|_{2},  \\
	& \hspace{1.7cm} \sup_{\gamma_{0,1}\in m_{0,1}} \inf_{\gamma_{0,1}\in m_{0,1} \gamma_{0,2}\in m_{0,2}} \|\theta( \gamma_{0,2}-\gamma_{0,1})\|_{2}\biggr\} \\
	= &\, \sup_{\gamma_{0,2}\in m_{0,2}} \inf_{\gamma_{0,1} \in m_{0,1}}\|\theta( \gamma_{0,2}-\gamma_{0,1})\|_{2} \\
	= &\, \theta d_H(m_{0,2}, m_{0,1})
	\end{align*}
	where we have used that the second term in the maximum is zero since the two models intersect. The penultimate equality holds by convexity of the models. 
	
	The same argument with $m_{0,2}$ in place of $m_{0,1}$ gives the second statement.
\end{proof}

\begin{proof}[Proof of Proposition \ref{propMixDist}]
	For notational convenience, $m_j(\varepsilon)$ denotes $m_j \cap N_{\varepsilon}(\gamma^*)$. The triangle inequality and the definition of Hausdorff distance in Definition \ref{defCEquiv} gives 
	\begin{align*}
	d_H(m_0(\varepsilon), m_1(\varepsilon)) \leq &\, d_H(m_0(\varepsilon), m_{0,1}(\varepsilon)) + d_H(m_{0,1}(\varepsilon) ,m_1(\varepsilon))\\
	\leq &\, d_H(m_0(\varepsilon), m_{0,1}(\varepsilon)) + d_H(m_{0,1}(\varepsilon), m_{0,2}(\varepsilon)) + d_H(m_{0,2}(\varepsilon), m_1(\varepsilon))\\
	= &\, o(\max\{\varepsilon^{c}, \varepsilon^{c_2}\}),
	\end{align*}
	where we have used, from Lemma \ref{lemmaHausDistM01}, that $d_H(m_0(\varepsilon),m_{0,1}(\varepsilon)) = \theta d_H(m_{0,1}(\varepsilon), m_{0,2}(\varepsilon)) = O(\varepsilon)o(\varepsilon^c) = o(\varepsilon^{c+1})$. The last line follows on observing that $\varepsilon^{c+1} < \varepsilon^c$. By the same argument, 
	\begin{align*}
	d_H(m_0(\varepsilon), m_1(\varepsilon)) \leq &\, d_H(m_0(\varepsilon), m_{0,2}(\varepsilon)) + d_H(m_{0,2}(\varepsilon) ,m_1(\varepsilon))\\
	\leq &\, d_H(m_0(\varepsilon), m_{0,2}(\varepsilon)) + d_H(m_{0,2}(\varepsilon), m_{0,1}(\varepsilon)) + d_H(m_{0,1}(\varepsilon), m_1(\varepsilon))\\
	= &\, o(\max\{\varepsilon^{c}, \varepsilon^{c_1}\}).
	\end{align*}
	The result stated in the proposition is the minimum from these two statements.
\end{proof}

\begin{proof}[Proof of Corollary \ref{corMix}]
	The proof of (i) follows that of Lemma \ref{lemmaCS}. Introduce parameter spaces $\Gamma_0, \Gamma_{0,1}, \Gamma_{0,2}$ such that $\gamma_0 \in \Gamma_0$ corresponds to the true mixture model $m_0$ and $\gamma_{0,j} \in \Gamma_{0,j}$ ($j=1,2$) corresponds to the permissible parameters under the postulated model $m_1$, which in the case of part (i) is $m_{0,1}$. As in the proof of Lemma A.1, introduce local parameter spaces corresponding to $\Gamma_0, \Gamma_{0,1}, \Gamma_{0,2}, \Gamma$ and suppose that they converge to sets $H_0, H_{0,1}, H_{0,2}$ and $H$. The parameter associated with the true mixture is 
	\begin{align*}
	\gamma_0 = &\, 
	\gamma_{0,1} + \theta(\gamma_{0,2} - \gamma_{0,1})
	=  \gamma_{0,1} + \theta (h_{0,2} - h_{0,1}).
	\end{align*}
	Thus, on setting $\theta = n^{-1/2}$ and $h_0 = h_{0,2} - h_{0,1}$ the problem of interest has the same contiguous alternatives form $\gamma_{0,1} + h_0/\sqrt{n}$ as in the proof of Lemma \ref{lemmaCS}.
	
	The likelihood ratio statistic  $\Lambda_n$ for testing the postulated model $m_1 = m_{0,1}$ against the saturated one converges in distribution under $\gamma_{0,1} + h_0/\sqrt{n}$  to $\Lambda$, where 
	\begin{align*}
	\Lambda &= \inf_{h_{0,1}\in H_{0,1}}(X-h_{0,1})^\T I_{\gamma_0}(X-h_{0,1}) - \inf_{h_{0,1}\in H}(X-h_{0,1})^\T I_{\gamma_0}(X-h_{0,1}) \\ 
	&=: \| I_{\gamma_0}^{1/2}X - I_{\gamma_0}^{1/2}H_{0,1}\|^2 - \|I_{\gamma_0}^{1/2}X - I_{\gamma_0}^{1/2}H\|^2, \label{Lambda2}
	\end{align*}
	and the rest of the argument in the proof of Lemma \ref{lemmaCS} applies unmodified. 
	
	The proof of (ii) uses a different argument based on the $c$-equivalence. From Lemma \ref{lemmaHausDistM01}, $d_H(m_0(\varepsilon), m_{0,2}(\varepsilon)) = (1-\theta)d_H(m_{0,1}(\varepsilon),m_{0,2}(\varepsilon))$ and $m_{0,1}$ and $m_{0,2}$ are $c-$equivalent, so  $d_H(m_{0,1}(\varepsilon),m_{0,2}(\varepsilon)) = o(\varepsilon^c)$, implying that $d_H(m_0(\varepsilon), m_{0,2}(\varepsilon)) = o(\varepsilon^c)$. We have shown that that the true mixture model $m_0^{(n)}$ and the postulated $m_{0,2}^{(n)}$ are $c$-equivalent. An application of Proposition \ref{propCSProperties} establishes the result. 
\end{proof}

\begin{proof}[Proof of Corollary \ref{corMix2}]
	From Proposition \ref{propMixDist}, the true and postulated models are $c_m$-equivalent in a neighbourhood of $\gamma^*$, where
	\[
	c_m := \max\{\min\{ c, c_1\},\min\{ c, c_2\}\}.
	\]
	Because $m_{0,1}^{(n)}, m_{0,2}^{(n)}$ and $m_1^{(n)}$ are mean-zero Gaussian models with non-singular covariance matrices, they are infinitely differentiable. Since any the convex combination of $m_{0,1}^{(n)}$ and $m_{0,2}^{(n)}$ preserves this differentiability, the true mixture $m_0^{(n)}$ is also infinitely differentiable. Thus, $c_m$-equivalence implies $(c_m+1)$-near-equivalence as noted by \citet{Evans2020}. The notion of $c$-near-equivalence gives sharper results than $c$-equivalence.

	An application of Proposition \ref{propCSProperties} with $m_0^{(n)}$ and $m_1^{(n)}$, $(c_m+1)$-near-equivalent shows that $m_1^{(n)}$ is retained in $\mathcal{M}$ with probability $1-\alpha + o(1)$ if and only if $a_{0,1}^{(n)} \asymp a_{0,2}^{(n)} \asymp a_1^{(n)} = O(n^{1/2c_m})$, where $c_m = c$ if $c \leq \max\{c_1, c_2\}$ and $c_m = \max\{c_1, c_2\}$ if $\max\{c_1, c_2\} \leq c$. 
\end{proof}

\section{Additional plots}

To illustrate Proposition \ref{propCSProperties} from a different perspective, we present the simulated retention probabilities of a 95\% confidence set, from 500 replications, for the 12 models in Table \ref{tabVConfig} when the true distribution belongs to model 1 with varying signal strength. Models 2--4 are 2-near-equivalent to model 1 as they share the same skeleton, and the remaining models are 1-near-equivalent. 

From Proposition \ref{propCSProperties}, 1-near-equivalent models are indistinguishable from the true one if and only if the signal strength is $o(n^{-1/2})$, whereas 2-near-equivalent models are indistinguishable from the true one if and only if the signal strength is $o(n^{-1/4})$. These conclusions are reflected in the plots. That 1-near-equivalent and 2-near-equivalent models are indistinguishable at any sample size if the signal decays faster than $n^{-1/2}$ is apparent from Figure \ref{fig:scaling n-1}. When the signal strength is at the borderline of detectability, $n^{-1/2}$, as in Figure \ref{fig:scaling n-1/2}, the 1-near-equivalent models are retained less frequently than the nominal rate, but as the sample size is increased, the corresponding retention probabilities do not approach zero; the 2-near-equivalent models are retained with simulated probabilities close to $1-\alpha$. With a signal strength of $n^{-1/4}$ in Figure \ref{fig:scaling n-1/4}, the retention probabilities for the 1-near-equivalent models approach zero with increasing sample size, while those for the 2-near-equivalent models are stable but below the nominal coverage rate $1-\alpha$. At a slowly decaying signal strength of $n^{-1/6}$, both 1-near-equivalent and 2-near-equivalent models are distinguishable from the true one. The figures show that a stronger signal strength, i.e.~one that decays more slowly with sample size, is needed to exclude models from the confidence set that are locally similar to the true one, as characterised by a larger $c$ in their $c-$equivalence to model $m_0$.

\begin{figure}
	\vspace{-0.2cm}
	\centering
	\begin{subfigure}[b]{0.49\textwidth}
		\centering
		\includegraphics[trim=0.0in 0.2in 0.0in 1in, clip, height=0.2\paperwidth]{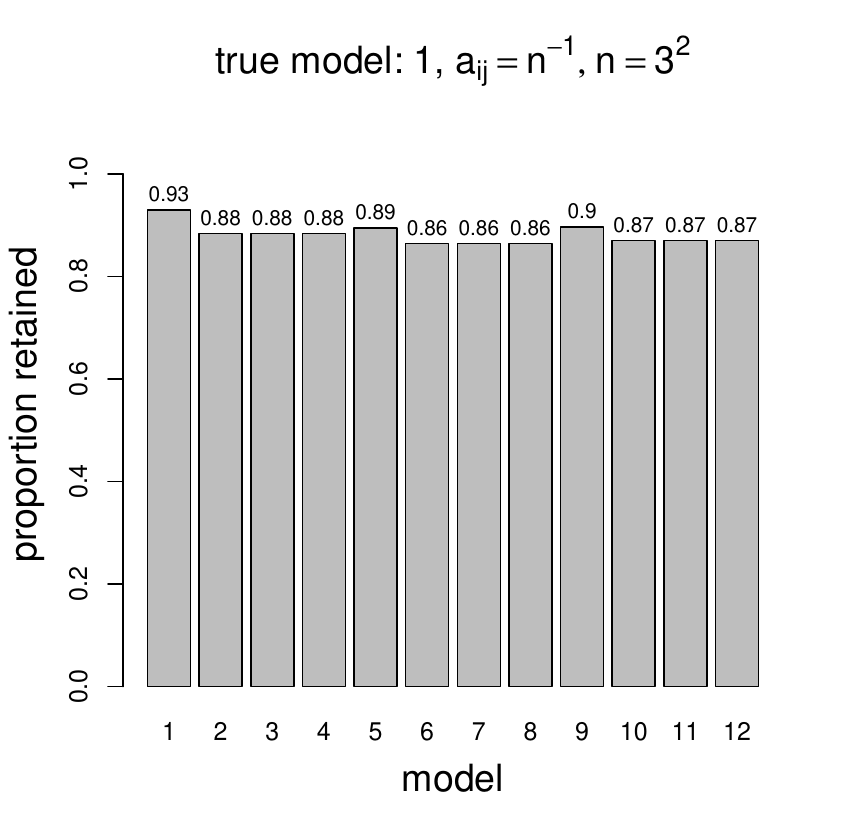}
	\end{subfigure}
	\hspace{-0.8cm}
	\begin{subfigure}[b]{0.49\textwidth}
		\centering
		\includegraphics[trim=0.0in 0.2in 0.0in 1in, clip, height=0.2\paperwidth]{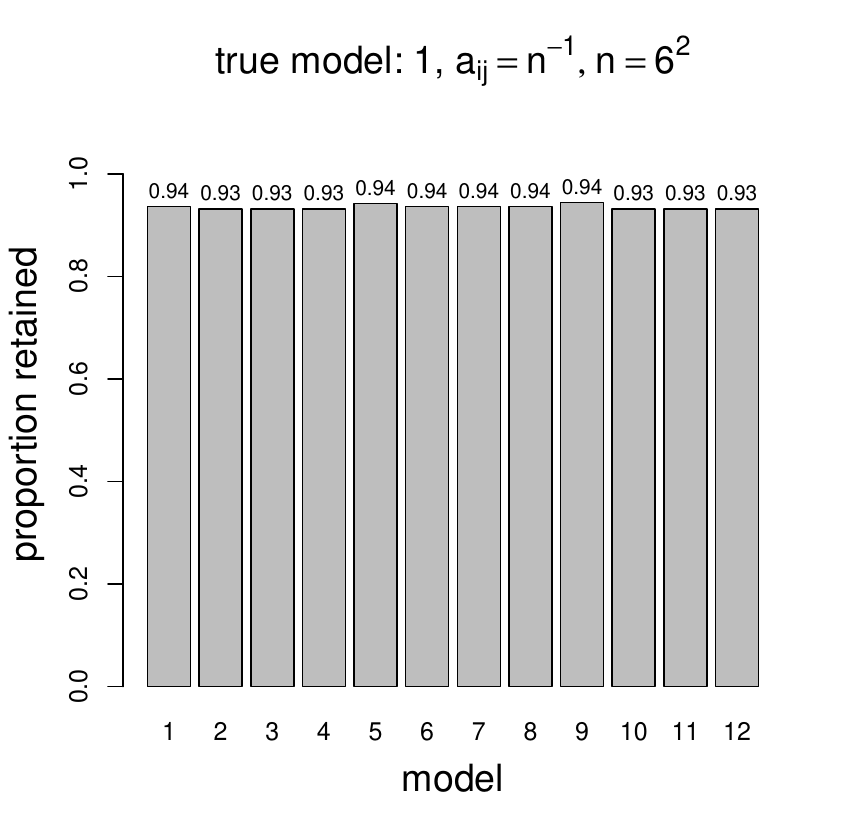}
	\end{subfigure} 
	\\[-0ex]
	\vspace{0.2cm}
	\par
	\centering
	\begin{subfigure}[b]{0.49\textwidth}
		\centering
		\includegraphics[trim=0.0in 0.2in 0.0in 1in, clip, height=0.2\paperwidth]{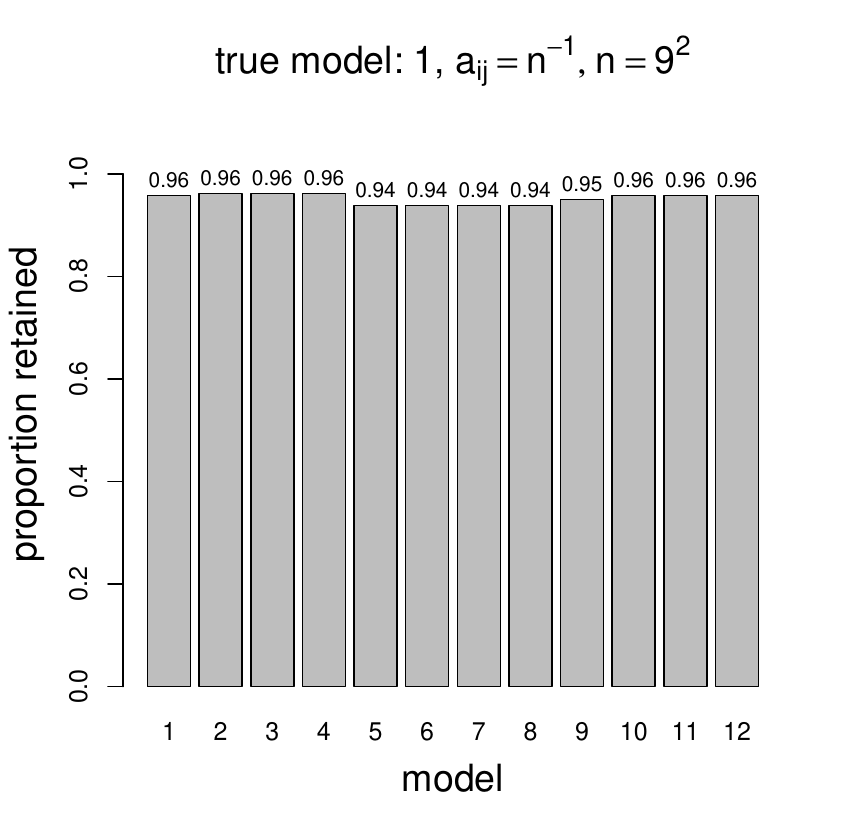}
	\end{subfigure}
	\hspace{-0.8cm}
	\begin{subfigure}[b]{0.49\textwidth}
		\centering
		\includegraphics[trim=0.0in 0.2in 0.0in 1in, clip, height=0.2\paperwidth]{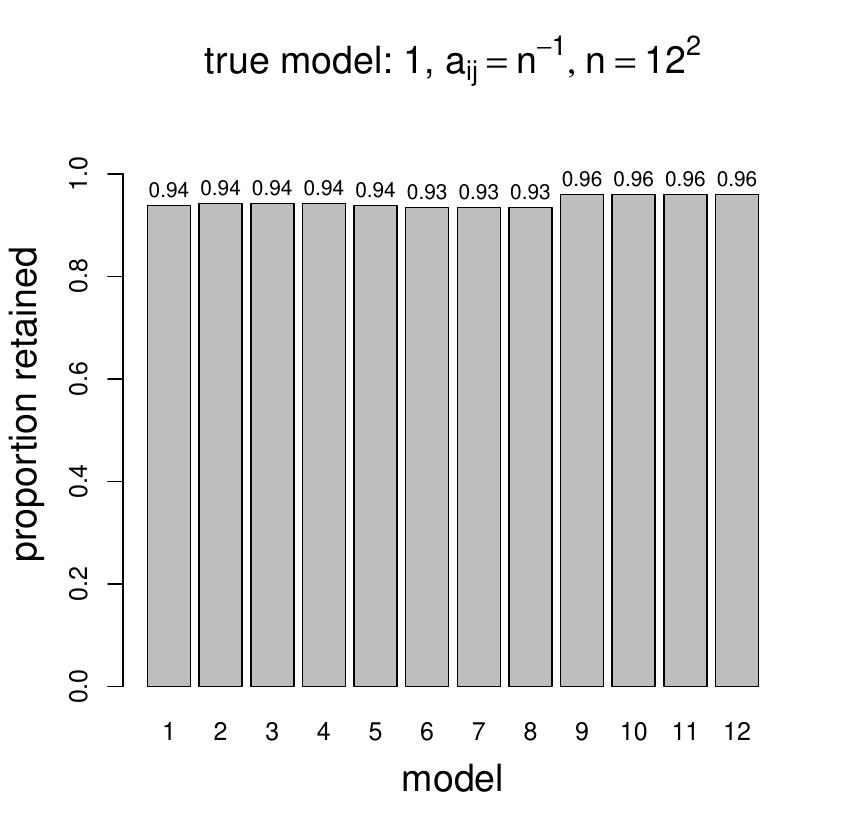}
	\end{subfigure} 
	\caption{Simulated retention probabilities for each model with increasing $n$ when the true distribution belongs to model 1 with signal strengths $a_{ij}=n^{-1}$. Top left to bottom right illustrates effect of increasing $n\in\{3^2,6^2,9^2,12^2\}$. \label{fig:scaling n-1} }
\end{figure}

\begin{figure}[h!]
	\vspace{-0.2cm}
	\centering
	\begin{subfigure}[b]{0.49\textwidth}
		\centering
		\includegraphics[trim=0.0in 0.2in 0.0in 1in, clip, height=0.2\paperwidth]{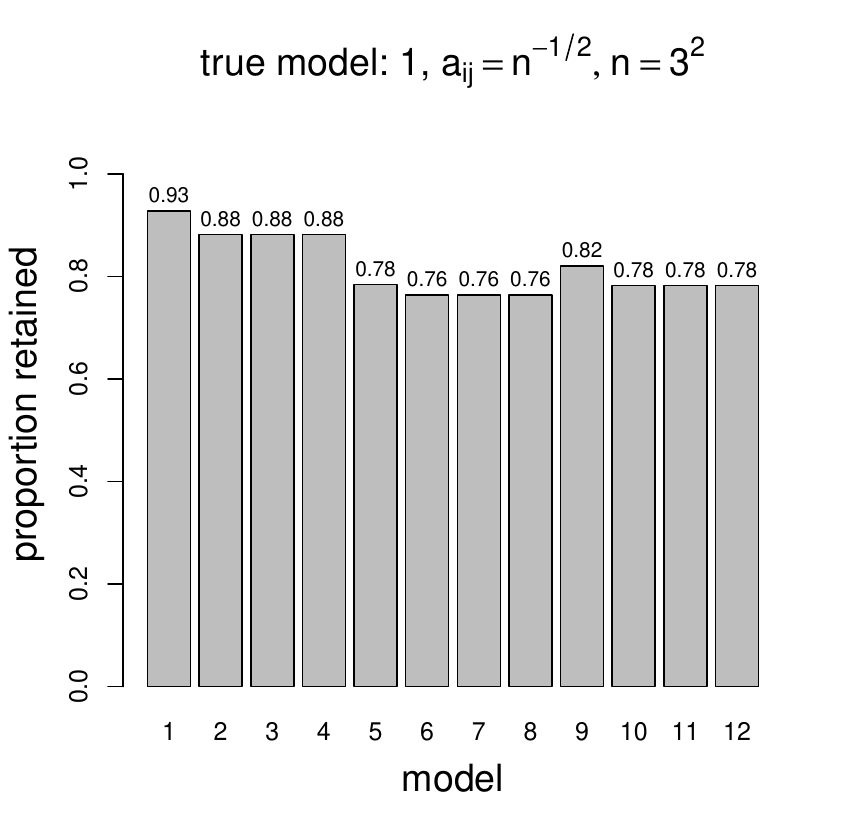}
	\end{subfigure}
	\hspace{-0.8cm}
	\begin{subfigure}[b]{0.49\textwidth}
		\centering
		\includegraphics[trim=0.0in 0.2in 0.0in 1in, clip, height=0.2\paperwidth]{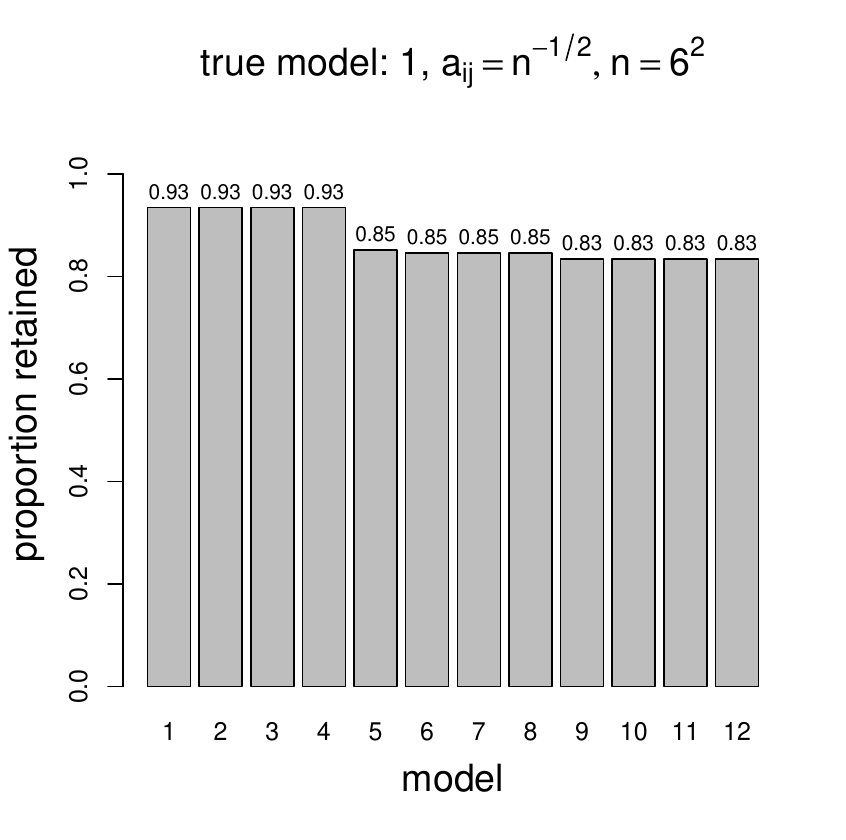}
	\end{subfigure} 
	\\[-0ex]
	\vspace{0.2cm}
	\par
	\centering
	\begin{subfigure}[b]{0.49\textwidth}
		\centering
		\includegraphics[trim=0.0in 0.2in 0.0in 1in, clip, height=0.2\paperwidth]{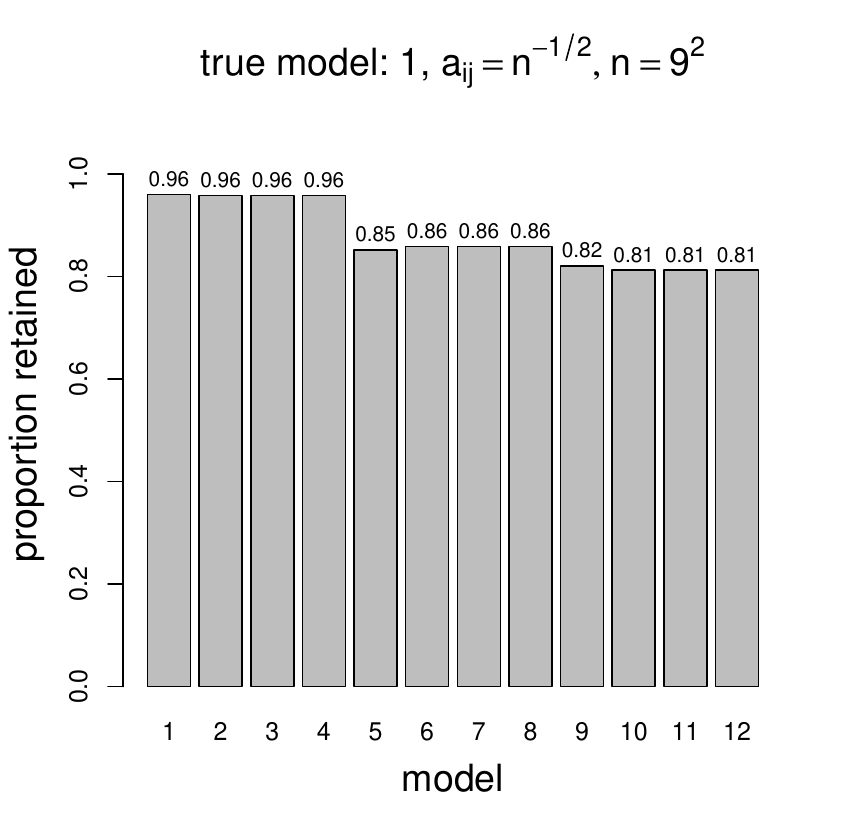}
	\end{subfigure}
	\hspace{-0.8cm}
	\begin{subfigure}[b]{0.49\textwidth}
		\centering
		\includegraphics[trim=0.0in 0.2in 0.0in 1in, clip, height=0.2\paperwidth]{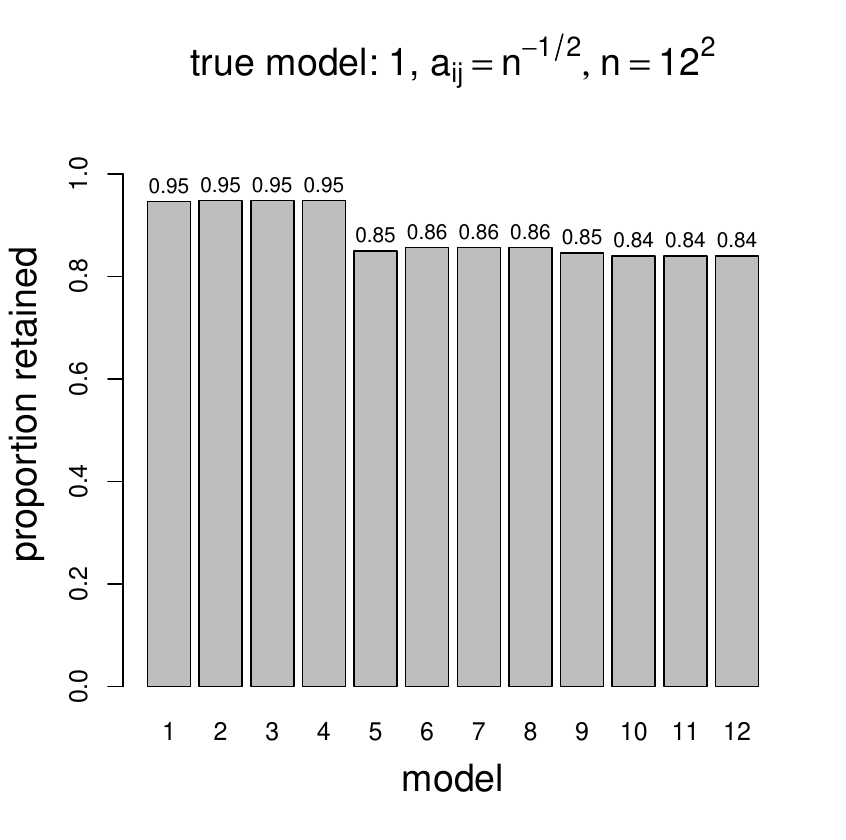}
	\end{subfigure} 
	\caption{Analogue of Figure \ref{fig:scaling n-1} with signal strengths $a_{ij}=n^{-1/2}$. \label{fig:scaling n-1/2} }
\end{figure}

\begin{figure}[h!]
	\centering
	\begin{subfigure}[b]{0.49\textwidth}
		\centering
		\includegraphics[trim=0.0in 0.2in 0.0in 1in, clip, height=0.2\paperwidth]{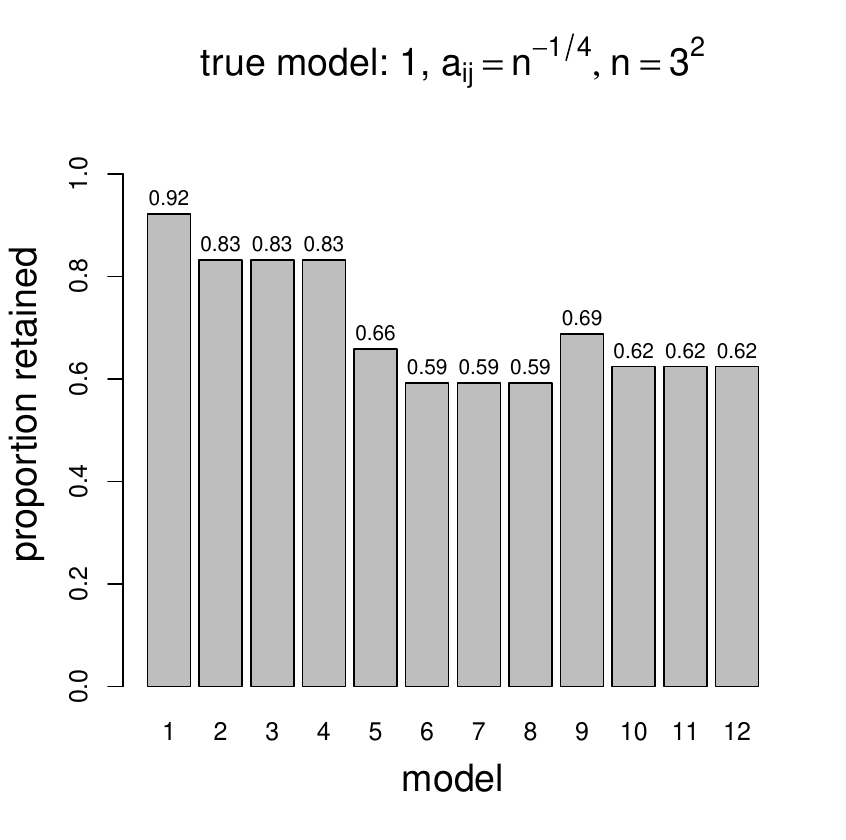}
	\end{subfigure}
	\hspace{-0.8cm}
	\begin{subfigure}[b]{0.49\textwidth}
		\centering
		\includegraphics[trim=0.0in 0.2in 0.0in 1in, clip, height=0.2\paperwidth]{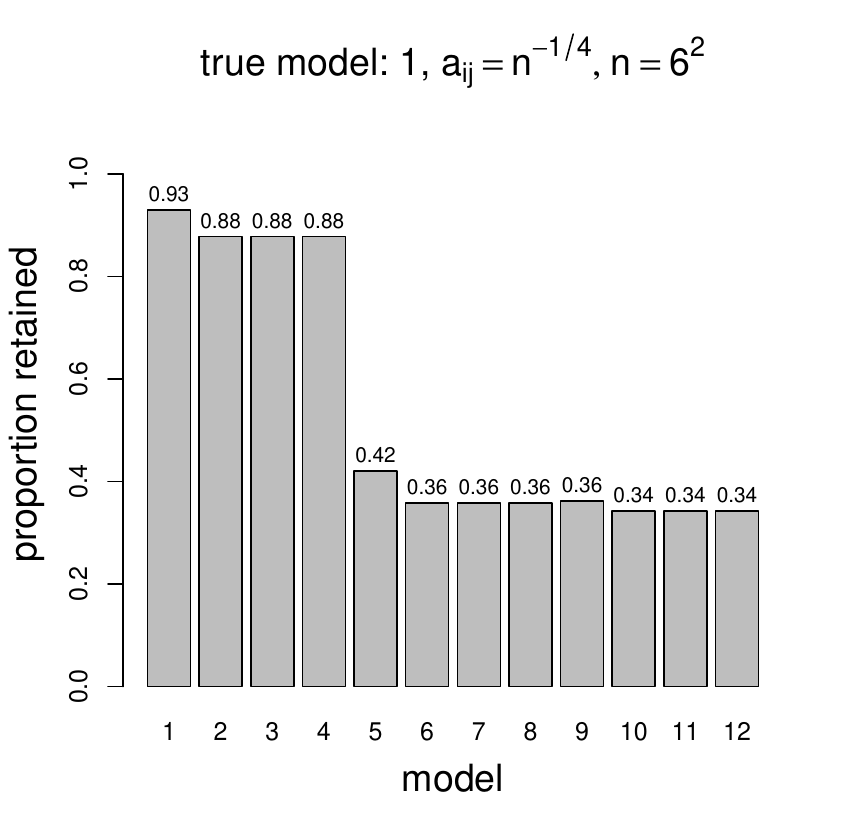}
	\end{subfigure} 
	\\[-0ex]
	\vspace{0.2cm}
	\par
	\centering
	\begin{subfigure}[b]{0.49\textwidth}
		\centering
		\includegraphics[trim=0.0in 0.2in 0.0in 1in, clip, height=0.2\paperwidth]{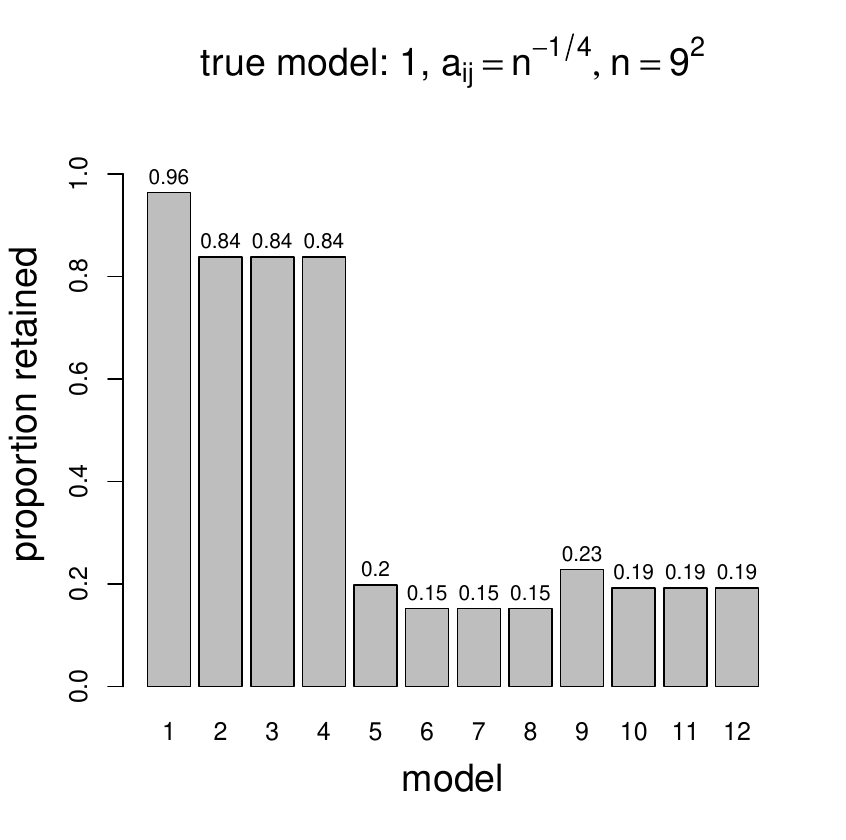}
	\end{subfigure}
	\hspace{-0.8cm}
	\begin{subfigure}[b]{0.49\textwidth}
		\centering
		\includegraphics[trim=0.0in 0.2in 0.0in 1in, clip, height=0.2\paperwidth]{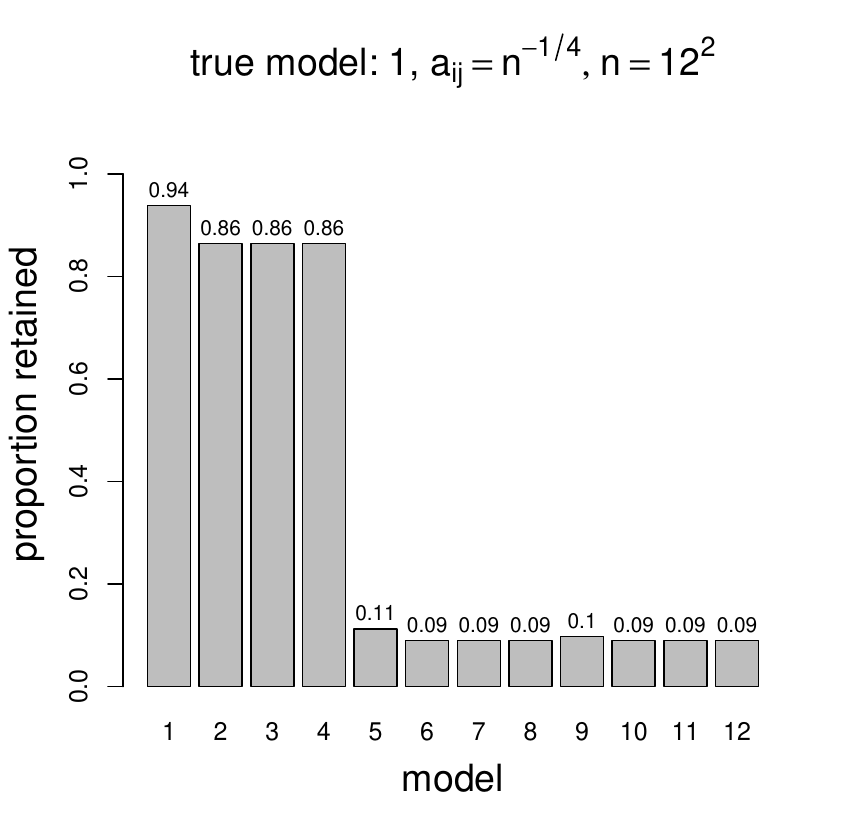}
	\end{subfigure} 
	\caption{Analogue of Figure \ref{fig:scaling n-1} with signal strengths $a_{ij}=n^{-1/4}$.  \label{fig:scaling n-1/4} }
\end{figure}

\begin{figure}[h!]
	\vspace{-0.2cm}
	\centering
	\begin{subfigure}[b]{0.49\textwidth}
		\centering
		\includegraphics[trim=0.0in 0.2in 0.0in 1in, clip, height=0.2\paperwidth]{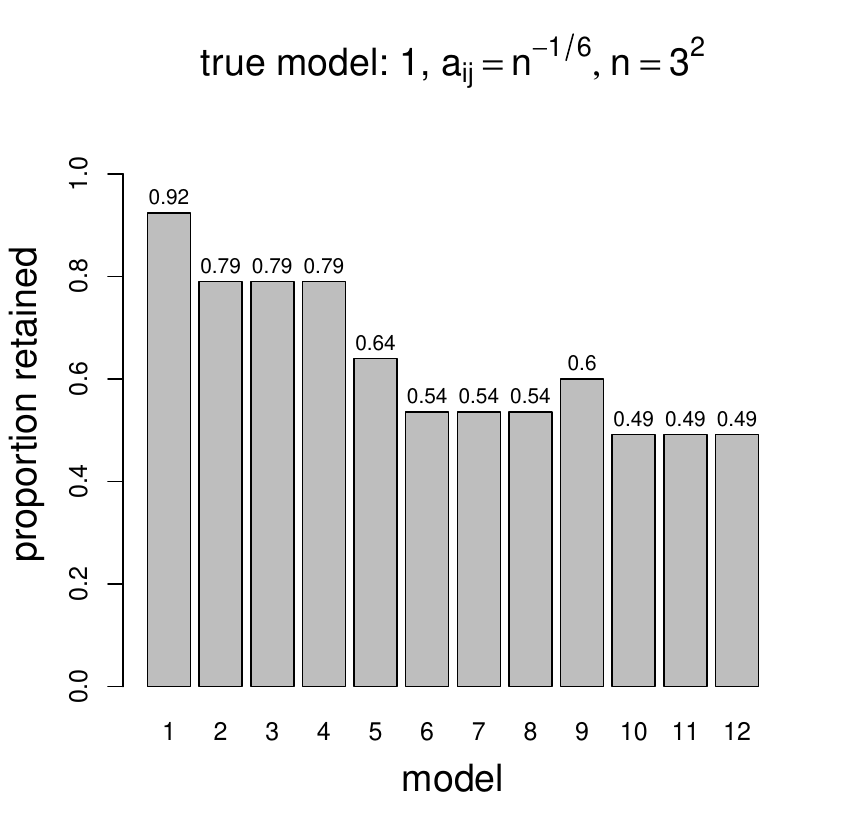}
	\end{subfigure}
	\hspace{-0.8cm}
	\begin{subfigure}[b]{0.49\textwidth}
		\centering
		\includegraphics[trim=0.0in 0.2in 0.0in 1in, clip, height=0.2\paperwidth]{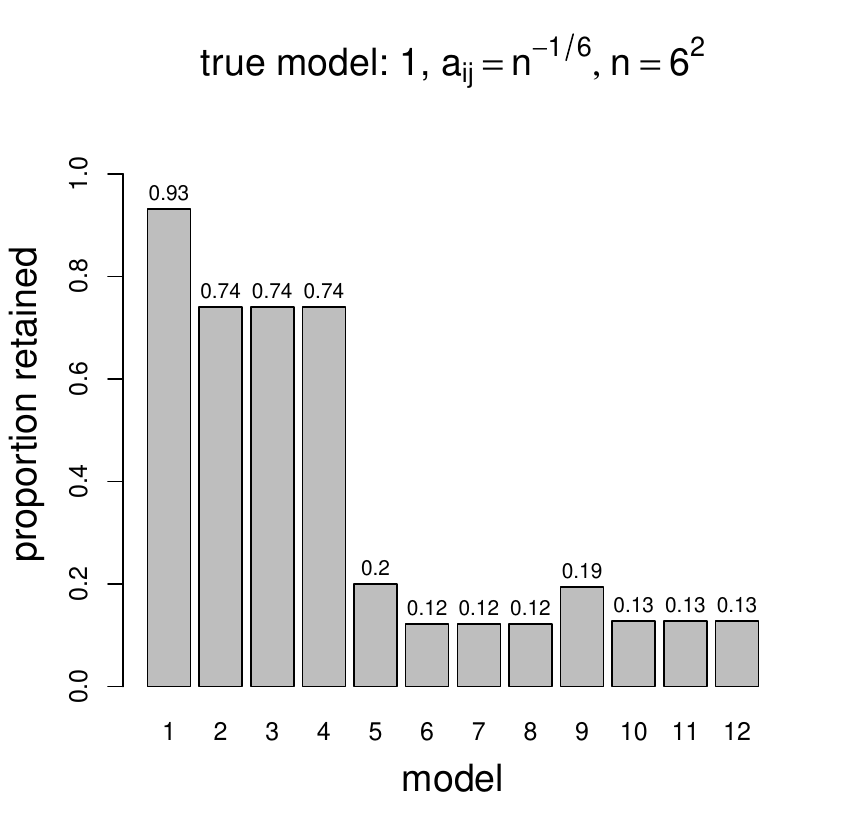}
	\end{subfigure} 
	\\[-0ex]
	\vspace{0.2cm}
	\par
	\centering
	\begin{subfigure}[b]{0.49\textwidth}
		\centering
		\includegraphics[trim=0.0in 0.2in 0.0in 1in, clip, height=0.2\paperwidth]{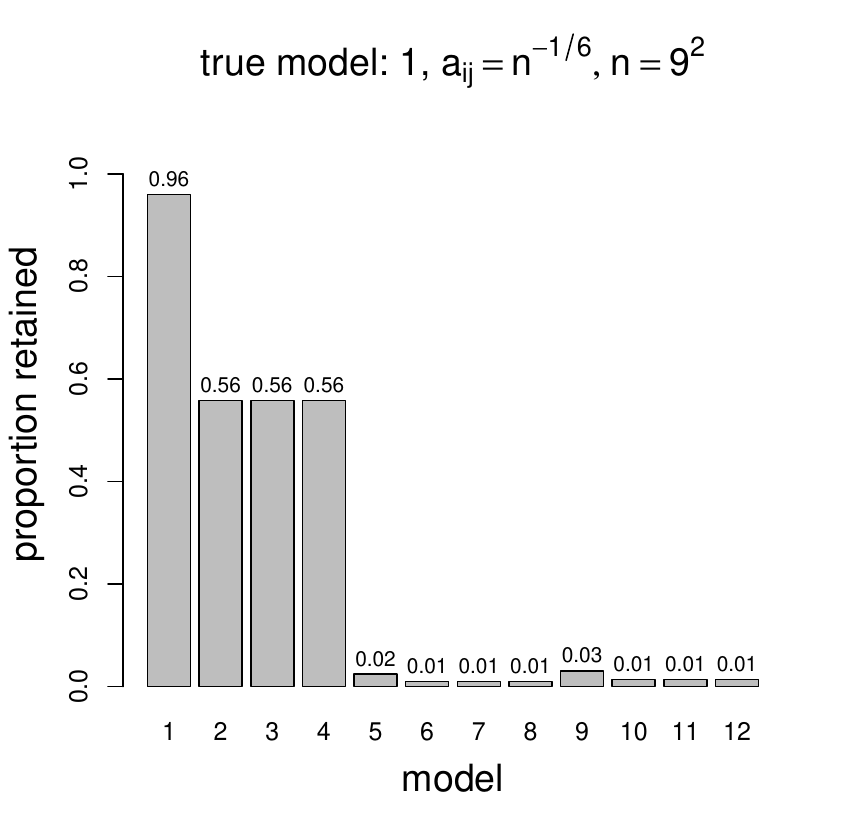}
	\end{subfigure}
	\hspace{-0.8cm}
	\begin{subfigure}[b]{0.49\textwidth}
		\centering
		\includegraphics[trim=0.0in 0.2in 0.0in 1in, clip, height=0.2\paperwidth]{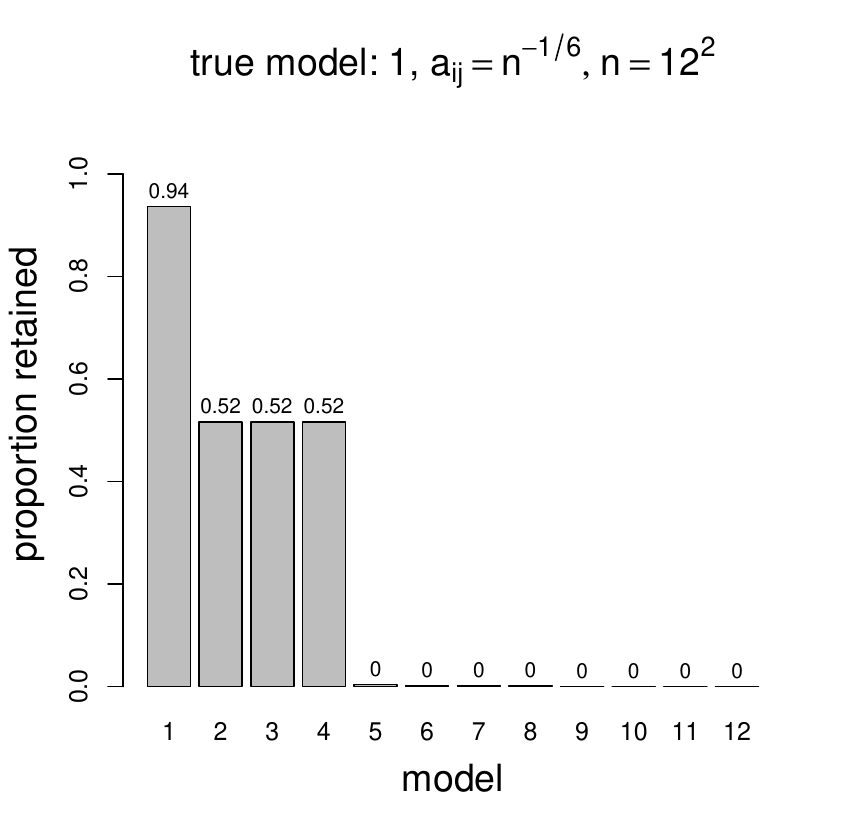}
	\end{subfigure} 
	\caption{Analogue of Figure \ref{fig:scaling n-1} with signal strengths $a_{ij}=n^{-1/6}$. \label{fig:scaling n-1/6} }
\end{figure}

\end{appendix}

\end{document}